\documentclass[12pt]{amsart}

\usepackage{amssymb,latexsym,amsxtra,amscd,ifthen}

\usepackage{amsmath}

\usepackage{amsfonts}

\usepackage{verbatim}

\usepackage{dsfont}

\usepackage{mathtools}

\usepackage{enumerate}

\usepackage[capitalise]{cleveref}

\usepackage{color}

\numberwithin{equation}{section}

\theoremstyle{plain}

\newtheorem{theorem}{Theorem}[section]

\newtheorem{lemma}[theorem]{Lemma}

\newtheorem{conjecture}[theorem]{Conjecture}

\newtheorem{claim}[theorem]{Claim}

\newtheorem{proposition}[theorem]{Proposition}

\newtheorem{corollary}[theorem]{Corollary}

\theoremstyle{definition}

\newtheorem{definition}[theorem]{Definition}

\theoremstyle{remark}

\DeclareMathAlphabet{\mathbbold}{U}{bbold}{m}{n}
\def\bb1{\mathbbold{1}}

\def\bbn{\mathbb{N}}

\DeclareMathOperator\Id{Id}

\DeclareMathOperator\Aut{Aut}

\DeclareMathOperator\ad{ad}
\DeclareMathOperator\Lie{Lie}

\DeclareMathOperator\End{End}

\DeclareMathOperator\spn{Span}

\newcommand{\til}[1]{\widetilde{#1}}

\newtheorem*{remark*}{Remark}

\usepackage{xparse}
\DeclareDocumentCommand{\adx}{ O{2} O{x_1}  }{\ad_{#2}^{[#1]}}

\begin{document}
\title{Lie algebras and torsion groups with identity}
\author{E. Zelmanov}
	
	\keywords{the Burnside problem, pro-p groups, PI-algebras, Lie algebras}
	
	\subjclass[2010]{Primary: 20E18, 20F50, 20F40, 16R99}
	
	\dedicatory{To my teacher Leonid A. Bokut on his 80th birthday.}

	\address{Efim Zelmanov\\		
		ezelmano@math.ucsd.edu\\		
		Department of Mathematics, University of California, San Diego\\
		 9500 Gilman Drive, La Jolla, CA 92093-0112, U.S.A.}
	
	\begin{abstract}
	We prove that a finitely generated Lie algebra $L$ such that (i) every commutator in generators is ad-nilpotent, and (ii) $ L$ satisfies a polynomial identity,   is nilpotent. As a corollary we get that a finitely generated residually-$p$ torsion group whose pro-$p$ completion satisfies a pro-$p$ identity  is finite. 
	\end{abstract}

\maketitle

\section{introduction}

In 1941 A.G Kurosh formulated a Burnside-type problem for algebras \cite{Ku}.  Let $A$ be an associative algebra over a field $F$.  An element $a\in A$ is said to be nilpotent if $a^{n(a)}=0$ for some $n(a)\geq1$.  An algebra $A$ is said to be nil if every element of $A$ is nilpotent.

\underline{The Kurosh Problem}: Is it true that a finitely generated nil algebra is nilpotent?

Examples by E. S. Golod \cite{G} (see also the far reaching examples from \cite{LS}) showed that this is not always the case.  However the Kurosh Problem has positive solution in the class of algebras satisfying a polynomial identity (PI-algebras).

Let $f(x_1,x_2\cdots, x_m)$ be a nonzero element of the free associative $F$-algebra.  We say that an algebra $A$ satisfies the polynomial identity $f=0$ if $f(a_1,a_2,\cdots,a_m)=0$ for arbitrary elements $a_1,a_2,\cdots, a_m \in A$.

One of the high points of the theory of PI-algebras was the solution of the Kurosh Problem (I. Kaplansky \cite{Kap}, J. Levitzki \cite{L}, A. I. Shirshov \cite{Sh}) in the following form:

Let $A$ be an associative algebra generated by elements $a_1, \cdots,a_m$.  Let $S$ be the multiplicative semigroup generated by the elements \linebreak$a_1,\cdots,a_m$.  Suppose that an arbitrary element of $S$ is nilpotent.  Then the algebra $A$ is nilpotent.

Now let $L$ be a Lie algebra over a field $F$.  As above, for a nonzero element $f(x_1,x_2,\cdots,x_m)$ of the free Lie algebra we say that $L$ satisfies the identity $f=0$ if $f(a_1,a_2,\cdots,a_m)=0$ for arbitrary elements $a_1,a_2,\cdots,a_m \in A$, see \cite{B}.

An element $a\in L$ is said to be \underline{ad-nilpotent} if the linear operator 
\[
ad(a):L\rightarrow L,
 x\rightarrow[x,a]\]
  is nilpotent.

A subset $S\subset L$ is called a \underline{Lie set} if, for arbitrary elements $a,b\in S$, we have $[a,b]\in S$.  For a subset $X\subset L$, the Lie set generated by $X$ is the smallest Lie set $S\langle X\rangle$ containing $X$.  It consists of $X$ and of all iterated commutators in elements from $X$.

\begin{theorem} \label{Theorem1}
	Let $L$ be a Lie algebra satisfying a polynomial identity and generated by elements $a_1, \cdots, a_m$.  If an arbitrary element $s\in S\langle a_1,\cdots, a_m\rangle$ is ad-nilpotent then the Lie algebra $L$ is nilpotent.
\end{theorem}

This theorem has implications in group theory:

Let $p$ be a prime number.  A group $G$ is said to be \underline{residually-p} if there exists a family of homomorphisms $\phi_i:G\rightarrow G_i$ into finite $p$-groups $G_i$ such that $\bigcap\limits_i Ker(\phi_i)=(1)$.

Let $\mathds{Z}_p$ be the field of order $p$.  Consider the group algebra\\ $(\mathds{Z}_p)[G]$ and its fundamental ideal $w$ spanned by all elements $1-g,g\in G$.  It is easy to see that the group $G$ is residually-$p$ if and only if $\bigcap\limits_{i\geq1}w^i=(0)$.  The \underline{Zassenhaus filtration} is defined as
\[G=G_1>G_2>\cdots\]
where $G_i=\{g\in G|1-g\in w^i\}$.  Then $[G_i,G_j]\subseteq G_{i+j}$ 
and each factor $G_i/G_{i+1}$ is an elementary abelian $p$-group.  Hence
\[L_p(G)=\bigoplus\limits_{i\geq 1}G_i/G_{i+1}\]
is a Lie algebra over $\mathds{Z}_p$.

\begin{theorem}\label{Theorem2} Let $G$ be a residually-$p$ finitely generated torsion group such that the Lie algebra $L_p(G)$ satisfies a polynomial identity.  Then $G$ is a finite group.
\end{theorem}

Let $g(x_1,x_2,\cdots,x_m)$ be a nonidentical element of the free pro-$p$ group (see \cite{S}, \cite{DSMS}) on the set of free generators $x_1,x_2,\cdots,x_m$.  We say that a pro-$p$ group $G$ satisfies the identity $g=1$ if \linebreak$g(a_1, a_2, \cdots,a_m)=1$ for arbitrary elements $a_1,a_2,\cdots,a_m\in G$.

\begin{theorem}\label{Theorem3}
	Let $G$ be a residually-$p$ finitely generated torsion group such that its pro-$p$ completion $G_{\hat{p}}$ satisfies a nontrivial identity.  Then $G$ is a finite group.	
\end{theorem}

Remark: the examples of infinite residually-$p$ groups due to E.S. Golod \cite{G}, R. I. Grigorchuk \cite{Gr}, and N. Gupta-S. Sidki \cite{GS} are finitely generated and torsion.

The results above significantly extend the positive solution of the Restricted Burnside Problem \cite{Z4, Z5} and the work of \cite{Z6} on compact torsion groups.  They were announced in \cite{Z7, Z8} but no detailed proof followed.  Meanwhile they were used in numerous papers.  Therefore I feel compelled to present a detailed proof.

The proof essentially uses the ideas and techniques from \cite{Z4, Z5}.

\begin{center}
\underline{Acknowledgments}
\end{center}

The author is grateful to the referees of this paper and to A. Fernandez Lopez for numerous valuable comments.

\section{The case of zero characteristic}
\label{Section2}

In this section we assume that char$F=0$.  This assumption allows us to avoid major difficulties but also miss major applications.

The following lemma is due to A. I. Kostrikin [Kos1; Kos2, Lemma 2.1.1].

\underline{Kostrikin Lemma}.  Let $L$ be a Lie algebra, $a\in L$, ad$(a)^n=0.$  If $4\leq n$ $\textless$ char$F$ (here zero characteristic is viewed as $\infty$), then 
\[\text{ad}(b\, \text{ad}(a)^{n-1})^{n-1}=0\]
 for an arbitrary element $b\in L$.
 
Choose a nonzero element $s\in S=S\langle a_1,\cdots, a_m\rangle$.  The element $s$ is ad-nilpotent.  Repeatedly using the Kostrikin lemma we can assume that ad$(s)^3=0$.

Recall that a linear algebra over a field $F$ of characteristic $\neq2$ is called a \underline{Jordan algebra} if it satisfies the identities
\begin{align*}
\text{(J1) } &x\circ y=y\circ x\\
\text{(J2) } & (x^2\circ y)\circ x = x^2\circ (y\circ x).
\end{align*}

If A is an associative algebra then $A^{(+)}=\{A, a\circ b=\displaystyle \frac{1}{2}(ab+ba)\}$ is a Jordan algebra.  For more information on Jordan algebras see \cite{J2,ZSSS, M3}.

We will use a construction of a Jordan algebra from \cite{FGG} which is a refined version of the Tits-Kantor-Koecher construction \cite{T1,T2,Kan,Ko}.

Let $L$ be a Lie algebra over a field of characteristic $\neq 2, 3$.  Let $s\in L$, ad$(s)^3=0$.  Define a new operation $a\circ b=[a,[s,b]]$, $a,b\in L$.  Then the vector space $K=\{a\in L|a$ ad$(s)^2=0\}$ is an ideal of the algebra $(L,\circ)$.

\begin{theorem}[\cite{FGG}]
	The factor algebra $(L,\circ)/K$ is a Jordan algebra.
\end{theorem}

For a set $X=\{x_1,x_2,\cdots\}$, let $FJ\langle X \rangle$ denote the free Jordan algebra (see \cite{J2, ZSSS, M3}).  Consider also the free associative algebra $F\langle X\rangle$.  Let $\phi$ be the homomorphism $\phi:FJ\langle X\rangle\rightarrow F\langle X \rangle^{(+)}$, $x\rightarrow x$, $x\in X$.  An element lying in the kernal $\ker\phi$, is called an $S$-identity.  A Jordan algebra $J$ is said to be $PI$ if there exists an element $f(x_1,\cdots,x_n)\in FJ\langle X\rangle$ that is not an $S$-identity such that $f(a_1,\cdots,a_n)=0$ for all elements $a_1,\cdots,a_n\in J$.

For elements $x, y, z$ of a Jordan algebra $J$, define their triple product $\{x,y,z\}=(xy)z+x(yz)-y(xz)$.  An element $a\in J$ is called an \underline{absolute zero divisor} if $a^2=0$ and $\{a,J,a\}=(0)$.

A Jordan algebra that does not contain nonzero absolute zero divisors is called \underline{nondegenerate}.  The smallest ideal $Mc(J)$ such that the factor algebra $J/Mc(J)$ is nondegenerate is called the McCrimmon radical of $J$.

\begin{lemma} \label{Lemma1}
	Let $J$ be a Jordan algebra with PI such that every element of $J$ is a sum of nilpotent elements.  Then $J=Mc(J)$.
	\begin{proof}
		Let $J\neq Mc(J)$.  Then without loss of generality we will assume that the algebra $J$ is nondegenerate.  Moreover, since a nondegenerate Jordan algebra is a subdirect product of prime nondegenerate Jordan algebras (see \cite{Z3}), we will assume that the algebra $J$ is prime and nondegenerate.
		
		In \cite{Z2} it was shown that a prime nondegenerate PI-algebra $J$ has nonzero \underline{center} 
\[Z(J)=\{z\in J|(za)b=z(ab) \text{ for arbitrary elements } a,b\in J\}\] 
and the ring of fractions $\til J=(Z(J)\setminus\{0\})^{-1}J$ is either a simple finite dimensional algebra over the field $\til Z=(Z(J)\setminus\{0\})^{-1}Z(J)$ or else an algebra of a symmetric nondegenerate bilinear form.  In both cases the algebra $\til J$ has a nonzero linear trace $t:\til J\rightarrow \til Z$ such that the trace of a nilpotent element is zero.  Since every element of $\til J$ is a sum of nilpotent elements it follows that $t(\til J)=(0)$, a contradiction that finishes the proof of the lemma.
	\end{proof}
\end{lemma}

\begin{lemma} \label{Lemma2}
	The Jordan algebra $J=(L,\circ)/K$ is McCrimmon radical, i.e. $J=Mc(J)$.
	\begin{proof}
		By our assumption, the Lie algebra $L$ satisfies a nontrivial polynomial identity.  Passing to the full linearization of this identity (see \cite{ZSSS}) we can assume that the identity looks like
		\[\sum\limits_{\sigma\in S_n}\alpha_\sigma x_0 \text{ ad}(x_{\sigma(1)})\cdots\text{ ad}(x_{\sigma(n)})=0,\]
		where not all coefficients $\alpha_\sigma\in F$ are equal to 0.  This implies that
		\[\sum\limits_{\sigma\in S_n} \alpha_\sigma a_0 R(a_{\sigma(1)})\cdots R(a_{\sigma(n)})=0\]
		for arbitrary elements $a_0, a_1.\cdots,a_n\in J$, where $R(a):x\rightarrow xa$ denotes the multiplication operator in $J$.		
		
		It is easy to see that the element $\sum\limits_{\sigma\in S_n} \alpha_\sigma a_0 R(a_{\sigma(1)})\cdots R(a_{\sigma(n)})$ is not an $S$-identity.  Hence $J$ is a PI-algebra.  
		
		The Lie algebra $L$ is spanned by the Lie set $S=S\langle a_1,\cdots a_m\rangle.$  For an arbitrary element $a\in S$ let $\bar{a}=a+K$ be its image in the Jordan algebra $J$. The t$^{th}$ power of $\bar{a}$ in $J$ is $a\text{ ad}([s,a])^{t-1}+K$, which implies that the element $\bar{a}$ is nilpotent.  By \Cref{Lemma1} $J=Mc(J)$, which finishes the proof of \Cref{Lemma2}.
	\end{proof}
\end{lemma}

Following A. I. Kostrikin\cite{K1, K3}, we call an element $a$ of a Lie algebra $L$ a \underline{sandwich} if (i) ad$(a)^2=0$ and (ii) ad$(a)\text{ ad}(b)\text{ ad}(a)=0$ for an arbitrary element $b\in L$.

If char$F\neq2$ then (i) implies (ii).

If $a$ is a nonzero absolute zero divisor of the Jordan algebra $J$ and char$F\neq2,3$ then for the nonzero element $b=a\text{ ad}(s)^2$ we have ad$(b)^2=\text{ad}(s)^2\text { ad}(a)^2\text{ ad}(s)^2.$  Hence $L\text{ ad}(b)^2\subseteq\{a,J,a\}\text{ ad}(s)^2=(0)$.  Hence, $b$ is a sandwich of the Lie algebra $L$.

To summarize, we showed that if $L=\langle a_1,\cdots,a_m\rangle$ is a nonzero Lie algebra over a field $F$ of zero characteristic, every element of the Lie set $S=S\langle a_1,\cdots ,a_m\rangle$ is ad-nilpotent, and if $L$ satisfies a nontrivial polynomial identity, then $L$ contains a nonzero sandwich.

\begin{lemma}	\label{Lemma3}
	Let $L=\langle a_1,\cdots,a_m\rangle$ be a finitely generated Lie algebra such that an arbitrary element of the Lie set $S\langle a_1,\cdots, a_m\rangle$ is ad-nilpotent.  Let $I$ be an ideal of $L$ of finite codimension.  Then $I$ is finitely generated as a Lie algebra.
	\begin{proof}
		The finite dimensional Lie algebra $L/I$ is spanned by a Lie set for which every element in the set is ad-nilpotent.  By the Engel-Jacobson theorem \cite{J1} the Lie algebra $L/I$ is nilpotent.  In other words, there exists $k\geq1$ such that $L^k\subseteq I$.
		
		Suppose that every commutator $\rho$ in $a_1\cdots,a_m$ of length $<k$ is ad-nilpotent of degree at most $t$, i.e. ad$(\rho)^t=0$.  Let $N=ktm^k$.  In \cite[Lemma 2.5]{Z8} it is shown that every product ad$(a_{i_1})\cdots \text{ad}(a_{i_N})$, $1\leq i_1,\cdots,i_N\leq m$, can be represented as \[\text{ad}(a_{i_1})\cdots\text{ad}(a_{i_N})=\sum\limits_j v_j\text{ ad}(\rho_j),\] where the $v_j$'s are (possibly empty) products of the $ad(a_i)$'s and the $\rho_j$'s are commutators in $a_1,\cdots,a_m$ of length $\geq$k.  Each summand on the right hand side has the same degree in each $a_i$ as the left hand side.
		
		It follows now that the algebra $L^k$ is generated by commutators $\rho$ in $a_1,\cdots,a_m$ such that $k\leq \text{length}(\rho)<2N$.
		
		We have dim$_F(L/L^k)<\infty$.  Let $b_1,\cdots,b_r\in I$ be a basis of $I$ modulo $L^k$.  Now the algebra $I$ is generated by $b_1,\cdots,b_r$ and by all commutators $\rho$ in $a_1,\cdots,a_m$ such that $k\leq \text{length}(\rho)<2N$, which proves the lemma.
	\end{proof}
\end{lemma}

Recall that an algebra $L$ is called \underline{just infinite} if it is infinite dimensional but every nonzero ideal of $L$ is of finite codimension.

\begin{lemma}
	\label{Lemma4} Let $L$ be an infinite dimensional Lie algebra generated by elements $a_1,\cdots,a_m$ such that an arbitrary element from $S\langle a_1,\cdots,a_m\rangle$ is ad-nilpotent.  Then $L$ has a just infinite homomorphic image.
	\begin{proof}
		Let $I_1\subseteq I_2\subseteq\cdots$ be an ascending chain of ideals of infinite codimension.  We claim that the union $I=\bigcup\limits_i I_i$ also has infinite codimension.  Indeed, if dim$_F(L/I)<\infty$ then by \Cref{Lemma3} the ideal $I$ is generated by a finite collection of elements, hence $I$ is equal to one of the terms in the ascending chain, a contradiction.
		
		By Zorn's Lemma the algebra $L$ has a maximal ideal $J$ of infinite codimension.  The factor algebra $L/J$ is just infinite, which proves the lemma.
	\end{proof}
\end{lemma}

Now we are ready to finish the proof of \Cref{Theorem1} in the case of char$F=0$.

Let $L$ be a Lie algebra satisfying the assumptions of the theorem.  In view of \Cref{Lemma4} without loss of generality we will assume the algebra $L$ to be just infinite. 

We proved that $L$ contains a nonzero sandwich.  Recall that an algebra is called locally nilpotent if every finitely generated subalgebra is nilpotent.  A. N. Grishkov  \cite{Gri} proved that in a Lie algebra over a field of zero characteristic, an arbitrary sandwich generates a locally nilpotent ideal.  Since the Lie algebra $L$ is just infinite it follows that $L$ contains a locally nilpotent ideal $I$ of finite codimension.  By \Cref{Lemma1} the algebra $I$ is finitely generated, hence nilpotent and finite dimensional.  This contradicts the assumption that the algebra $L$ is infinite dimensional and proves the theorem.

\section{Divided polynomials}\label{Section3}

The main Theorem \ref{Theorem1} is valid for Lie algebras over an arbitrary ground field $F$. The applications to Theorems \ref{Theorem2}, \ref{Theorem3} use only the case when the ground field $F$ is finite.

We will show that without loss of generality, we can assume that the field $F$ is infinite. Indeed, let $F'$ be an infinite field extension of $F$. The Lie algebra $L'=L\otimes_F F'$ is generated by the same elements $a_1, \cdots, a_m$ as $L$ and an arbitrary element $s\in S\langle a_1,\cdots, a_m\rangle$ is ad-nilpotent in $L'$. Since the Lie algebra $L$ satisfies a polynomial identity, it satisfies a nontrivial multilinear identity $f(x_1,\cdots, x_n)=0$ (see \cite{B}). Then the Lie algebra $L'$ also satisfies the identity $f=0$.

From now on we assume that char$F=p>0$ and the field $F$ is infinite.  Let $L\langle X\rangle$ be the free Lie $F$-algebra on the set of free generators $X=\{x_1,\cdots,x_m\}$ in the variety of algebras satisfying the identity $f=0$ (see \cite{B}).  Let $P$ be the set of all commutators in $X$ and let $n:P\rightarrow N$ be a function.  Let $J$ be the ideal of $L\langle X \rangle$ generated by $\bigcup\limits_{\rho\in P} L\langle X \rangle \text{ ad}(\rho)^{n(\rho)}$.

Our aim is to show that the algebra $L'=L\langle X \rangle/J$ is nilpotent.  Suppose that this is not true.  Letting $\deg(x_i)=1$, $1\leq i\leq m$, we define a gradation of $L'$ by positive integers.

We say that a graded infinite dimensional algebra is graded just infinite if every nonzero graded ideal of it is of finite codimension.

\begin{lemma}
	\label{Lemma5}The algebra $L'$ has a graded just infinite homomorphic image $L$.
\end{lemma}

The proof follows the proof of \Cref{Lemma4} (verbatim).

Consider the adjoint embedding $L\rightarrow End_F(L)$, $a\rightarrow \text{ad}(a)$.  Let $A'$ be the associative subalgebra of $End_F(L)$ generated by the image of $L$.  The algebra $A'$ is graded and we assume that $L\subseteq A'^{(-)}$.  Let $I$ be a maximal graded ideal of the algebra $A'$ such that $L\cap I=(0)$, $A=A'/I$, $L\subseteq A^{(-)}$.  If $J$ is a nonzero graded ideal of the algebra $A$ then the ideal $J\cap L$ has finite codimension in $L$.  From the Poincare-Birkhoff-Witt theorem it follows that the factor algebra $A/J$ is nilpotent and finite dimensional.  We have proved that the algebra $A$ is graded just infinite.

To summarize, we assume that 
\begin{enumerate}
	\item the graded Lie algebra $L$ is generated by elements $s_1,\cdots,s_m$ of degree 1; every element from the Lie set $S=S\langle s_1,\cdots,s_m\rangle$ is ad-nilpotent;
	\item $L$ satisfies a polynomial identity;
	\item $L$ is graded just infinite. 
\end{enumerate}
 We fix also a graded just infinite associative enveloping algebra $A$ of $L$.  The algebra $A$ is a homomorphic image of the subalgebra $\langle \text{ad}(L)\rangle\subseteq End_F(L)$
 
 For elements $a_1,\cdots,a_k\in L $ let $[a_1,a_2,\cdots,a_k]$ denote their left-normed commutator by $[ \cdots [a_1,a_2],a_3],\cdots,a_k]$  We also denote \linebreak $[a,\underbrace{b,b\cdots,b}_\text{k}]=[ab^k].$

 \begin{lemma}
	\label{Lemma6}  Let $I$ be an ideal of $L$, $s\in S$, $k \geq 2$, $[Is^k]=(0)$.  Suppose that the Lie algebra $L$ satisfies an identity of degree $n$.  Then the subalgebra $[Is^{k-1}]$ satisfies an identity of degree $<n$.
 	\begin{proof}
 		Let $L$ satisfy an identity 
 		\[\sum\limits_{\sigma\in S_{n-1}}\alpha_\sigma[x_0,x_{\sigma(1)},\cdots,x_{\sigma(n-1)}]=0,\]
 		where $\alpha_\sigma\in F$, $\alpha_1=1$. For the variables $x_0,x_1,\cdots,x_{n-1}$ choose values $x_0=s$, $x_1=a\in I$, $x_i=a_i\in [Is^{k-1}]$, $2\leq i\leq n-1$.  Then   \[\sum\limits_{\sigma\in H}\alpha_\sigma[[s,a],a_{\sigma(2)},\cdots,a_{\sigma(n-1)}]=0,\]
 		where $\sigma$ runs over the stabilizer $H$ of 1 in $S_{n-1}$.  It follows now that the Lie algebra $[Is^{k-1}]$ satisfies the identity
 		\[\sum\limits_{\sigma\in H}\alpha_\sigma[x_0,x_{\sigma(2)},\cdots,x_{\sigma(n-1)}]=0\]
 		of degree $n-1$.  This finishes the proof of the lemma.
 	\end{proof}
 \end{lemma}
 
 \begin{lemma}
	\label{Lemma7} Let $I$ be an ideal of $L$, $s \in L$, $[Is^k]=(0)$, $k \geq 2$. Then for any integer $t \geq 2$ and any elements $a_1, \cdots, a_N \in I$, $N=kt - 1$, we can write the operator $\text{ad}([a_1s^{k-1}]) \cdots \text{ad}([a_Ns^{k-1}])$ as a linear combination of operators of the type 
	\[P'\text{ad}(s)^{k-1}\prod_{j=0}^{t-2}(\text{ad}(a_{i+j})\text{ ad}(s)^{k-1})P''\]
	where $P',$ $P''$ are products of $\text{ad}(a_1), \cdots, \text{ad}(a_n), \text{ad}(s)$, which may be empty.
 \begin{proof}
 	By the Jacobi identity
 	\[\text{ad}([a_1s^{k-1}])\cdots\text{ad}([a_Ns^{k-1}]\]
	\[=\sum\pm\text{ad}(s)^{j_0}\text{ ad}(a_1)\text{ ad}(s)^{j_1}\cdots\text{ad}(s)^{j_{N-1}}\text{ ad}(a_N)\text{ ad}(s)^{j_N},\]
 	and in each summand $0\leq j_0, j_1,\cdots, j_N\leq k-1$, $j_0+\cdots+j_N=(k-1)N$. If, in each segment $j_\mu,j_{\mu+1}\cdots,j_{\mu+t-1}$ of length t, at least one term is $\leq k-2$, then $j_\mu+\cdots\jmath_{\mu+t-1}\leq(k-1)t-1$.  Summing all $k$ segments we get $j_0+\cdots+\jmath_N\leq k((k-1)t-1)<(k-1)(kt-1)$, a contradiction that proves the lemma.
 \end{proof}
  \end{lemma}
  
\begin{lemma}\label{Lemma8}
	There exist elements $c_1,\cdots,c_r\in S$ and integers $m\geq1, N\geq1$ such that
	\begin{enumerate}
		\item $[L^i,c_1,c_2,\cdots,c_r]\neq(0)$ for arbitrary $i\geq1$,
		\item $[L^m,c_1, c_2, \cdots, c_r,c_i]=(0)$ for $1\leq i\leq r$,
		\item for arbitrary elements $a_1,\cdots,a_N\in [L^m,c_1,\cdots ,c_r]$ we have ad$(a_1)\cdots\text{ad}(a_N)=0.$
	\end{enumerate}
	\begin{proof}
		Choose a nonzero element $s\in S$.  For an arbitrary $i\geq1$ choose a minimal integer $k(i)\geq1$ such that $[L^is^{k(i)}]=(0)$.  Since the Lie algebra $L$ is graded just infinite it follows that each power of $L$ has zero centralizer.  Hence $k(i)\geq2$.  We have $k(1)\geq k(2)\geq\cdots$.  There exists a sufficiently large integer $m_1$ such that
		\[ k_1:=k(m_1)=k(m_1+1)=\cdots.\]
		In other words, $[L^{m_1}s^{k_1}]=(0),[L^is^{k_1-1}]\neq(0)$ for any $i\geq1$.
		
		Now suppose that we have found $l$ elements $s_1=s,s_2,\cdots,s_l\in S$ and $2l$ integers $k_1,\cdots,k_l\geq 2;1\leq m_1\leq m_2\leq\cdots\leq m_l$ with the following properties:
		\begin{enumerate}
			\item $s_i\in [L^{m_{i-1}}s_1^{k_1-1}\cdots s_{i-1}^{k_{i-1}-1}],2\leq i\leq l$,
			\item $[L^{m_i}s_1^{k_1-1}\cdots s_{i-1}^{k_{i-1}-1} s_{i}^{k_i}]=(0),1\leq i\leq l$,
			\item for an arbitrary $i\geq1$ we have $[L^is_1^{k_1-1}\cdots s_{l}^{k_{l}-1}]\neq(0)$.
		\end{enumerate}
	\begin{claim}
		\label{Claim1}
	 For arbitrary $1\leq i,j\leq l$ we have $[s_i,s_j]=0$.
	 \begin{proof}
	Indeed, let $i<j$.  Then $s_j\in[L^{m_{j-1}}s_1^{k_1-1}\cdots s_{j-1}^{k_{j-1}-1}]$.  We will show that
	\[[L^{m_{j-1}}s_1^{k_1-1}\cdots s_{j-1}^{k_{j-1}-1}s_i]=(0).\]
	By the inductive assumption on $i+j$ the element $s_i$ commutes with $s_{i+1},\cdots,s_{j-1}$.  Hence \[[L^{m_{j-1}}s_1^{k_1-1}\cdots s_{j-1}^{k_{j-1}-1}s_i]=[L^{m_{j-1}}s_1^{k_1-1}\cdots s_i^{k_i}\cdots]=0,\] which proves the claim.\renewcommand{\qedsymbol}{$\blacksquare$}
	\end{proof}
\end{claim}
	
	\underline{Case 1.} Suppose that there exists an element $s'\in S^{m_l}=[\underbrace{S,S,\cdots,S]}_{m_l}$ such that
	\[ [[L^is_1^{k_1-1}\cdots s_{l}^{k_{l}-1}],[s's_1^{k_1-1}\cdots s_{l}^{k_{l}-1}]]\neq(0)\]
	for any $i\geq1.$  Denote $s_{l+1}=[s's_1^{k_1-1}\cdots s_{l}^{k_{l}-1}]$.  As above we find integers $m_{l+1}\geq m_l$ and $k_{l+1}\geq 2$ such that
	\[ [[L^{m_{l+1}}s_1^{k_1-1}\cdots s_{l}^{k_{l}-1}]s_{l+1}^{k_{l+1}}]=(0),\]
\[ [L^is_1^{k_1-1}\cdots s_{l}^{k_{l}-1}s_{l+1}^{k_{l+1}-1}]\neq(0)\]
	for any $i\geq 1$.  The elements $s_1,\cdots,s_{l+1}$ and the integers $m_1,\cdots,m_{l+1};$ $k_1,\cdots k_{l+1}$ satisfy the conditions 1), 2), and 3) above.
	
	\underline{Case 2} Now suppose that for an arbitrary element $s'\in S^{m_l}$ there exists an integer $i(s')\geq1$ such that
	\[ [[L^{i(s')}s_1^{k_1-1}\cdots s_{l}^{k_{l}-1}],[s's_1^{k_1-1}\cdots s_{l}^{k_{l}-1}]]=(0).\]
	Since $S^{m_l}$ spans $L^{m_l}$ it follows that for an arbitrary element $a\in [L^{m_l}s_{1}^{k_{1}-1}\cdots s_l^{k_l-1}]$ there exists $i(a)\geq1$ such that 
	\[ [[L^{i(a)} s_{1}^{k_{1}-1}\cdots s_l^{k_l-1}],a]=(0). \tag{1} \] 
	Let  $t=2k_1\cdots k_{l}-1$.  Choose $2t-1$ elements $a_1,\cdots,a_{2t-1}\in$ \break $[L^{m_l}s_{1}^{k_{1}-1}\cdots s_l^{k_l-1}]$.  Let $q=\text{max}\{i(a_\mu),1\leq \mu\leq2t-1,m_l\}$.
	
	Our immediate aim will be to show that 
	\[ [L^q,a_1,\cdots, a_t]=(0).\]
	Denote $t_j=2k_{j+1}\cdots k_l-1$, so $t_0=t$. We let $t_l=1$.  From \Cref{Claim1} it follows that for any $1\leq i\leq l$, $L_j:=[L^q s_{1}^{k_{1}-1}\cdots s_j^{k_j-1}]$ is a subalgebra of $L$.  Let $L_0=L^q$.
	
	\begin{claim}
		\label{Claim2}
		For $t_j$ arbitrary elements $b_1,\cdots,b_{t_j}\in \{a_1,\cdots, a_t\}$ we have 
		\[ [L_j,b_1,\cdots,b_{t_j}]=(0).\]
		\begin{proof}
		To prove the claim, we will use reverse induction on $j=0, \cdots, l$. For $j=l$ we have $t_l=1$ and $[L_l, a_i] = [L^qs_1^{k_1-1}\cdots s_l^{k_l-1}, a_i]=(0)$ by the choice of $q$. Now suppose that the assertion is true for $j$, $1 \leq j \leq l$. We have $t_{j-1}=k_j(t_j+1)-1.$ By Claim \ref{Claim1} an arbitrary element $a \in [L^q s_1^{k_1-1}\cdots s_l^{k_l-1}]$ can be represented as $a=[a's_j^{k_j-1}],$ where $a' \in [L^qs_1^{k_1-1}\cdots s_{j-1}^{k_1-1}] = L_{j-1}.$ Let $b_\mu=[b_\mu ' s_j^{k_j-1}],$ $b_\mu ' \in L_{j-1}.$ We apply Lemma \ref{Lemma7} to the algebra $L_{j-1}+Fs_j$ and its ideal $L_{j-1}$. By Lemma \ref{Lemma7} $\text{ad}(b_1)\cdots \text{ ad}(b_{t_{j-1}})$ is a linear combination of operators $P' \text{ad}(s_j)^{k_j-1}\big( \prod\limits_{\mu=0}^{t_j-1}\text{ad}(b'_{i+\mu}\text{ad}(s_j)^{k_j-1})\big) P''.$ By the induction assumption \break 
\[L_{j-1} P' \text{ad}(s_j)^{k_j-1}(\prod\limits_{\mu=0}^{t_j-1}\text{ad}(b'_{i+\mu})\text{ad}(s_j)^{k_j-1})\] 
\[\subseteq L_{j-1}\text{ ad}(s_j)^{k_j-1}(\prod\limits_{\mu=0}^{t_j-1} \text{ad}(b'_{i+\mu})\text{ad}(s_j)^{k_j-1}) \]
\[=L_j\prod\limits_{\mu =0}^{t-1}\text{ad}(b_{i+\mu }),\] 
which finishes the proof of Claim \ref{Claim2}.
		\end{proof}
	\end{claim}
	
	In particular, for $j=0$ we have 
	\[[L^q,a_1,a_2,\cdots,a_t]=(0).\]
	Now,
	\[ [L^q,[L,a_1,a_2,\cdots,a_{2t-1}]]\subseteq\sum[L^q,a_{i_1},\cdots,a_{i_\mu},L,a_{j_1},\cdots,a_{j_\nu}],\]
	where in each summand $\mu+\nu=2t-1$.  If $\mu\geq t$ then \[[L^q,a_{i_1},\cdots,a_{i_\mu}]=(0)].\]  If $\nu\geq t$ then 
	\[ [L^q,a_i,\cdots,a_{i_\mu},L,a_{j_1}\cdots,a_{j\nu}]\subseteq[L^q,a_{j_1},\cdots,a_{j_\nu}]=(0).\]
	Since the power $L^q$ has zero centralizer it follows that 
	\[[L,a_1,a_2,\cdots,a_{2t-1}]=(0).\]
	
	We showed that in case 2, the elements \[c_1,\cdots,c_r=\underbrace{s_1,\cdots,s_1}_{k_1-1},\underbrace{s_2,\cdots,s_2}_{k_2-1},\cdots,\underbrace{s_l,\cdots,s_l}_{k_l-1}\]
	and the integers $m=m_l,N=4k_1\cdots k_l-3$ satisfy the conditions of the lemma.

	Let the algebra $L$ satisfy an identity of degree $n$.  We will show that enlarging the system $s_1,\cdots,s_l$; $k_1\cdots,k_l\geq2$; $1\leq m_1\leq\cdots\leq m_l$ we will encounter case 2 in $\leq n-2$ steps.
	
	The subalgebra $I_i=[L^{m_i}s_1^{k_1-1}\cdots,s_{i-1}^{k_{i-1}-1}]$ is an ideal of 
	\[[L^{m_{i-1}}s_1^{k_1-1}\cdots,s_{i-1}^{k_{i-1}-1}]\] and $[I_is_i^{k_i}]=(0)$.  If $[L^{m_{i-1}}s_1^{k_1-1}\cdots s_{i-1}^{k_{i-1}-1}]$ satisfies an identity of degree $n_{i-1}$ then by \Cref{Lemma6} the algebra $[L^{m_i}s_1^{k_1-1}\cdots,s_{i-1}^{k_{i-1}-1}]$ satisfies an identity of degree $<n_{i-1}$.  This implies that $l\leq n-2$ and finishes the proof of the lemma.
	
\end{proof}
\end{lemma}

Let $E$ be the associative commutative $F$-algebra presented by the countable set of generators $e_i$, $i\geq1$, and relations $e_i^2=0$, $i\geq1$.  Ordered products $e_\pi=e_{i_1}\cdots e_{i_r}$, $\pi=\{i_1<i_2<\dots <i_r\}$ form a basis of the algebra $E$.  Notice that we don't consider empty products and therefore the algebra $E$ does not have 1.

We will start with a short \textbf{overview} of the rest of the proof of Theorm \ref{Theorem1}. The crucial role is played by the ``linearized" Lie algebra $\til{L}=L\otimes_F E$. In section \ref{Section3}, we define divided polynomials: a generalization of usual Lie polynomials that make sense in the context of the Lie algebra $\til{L}$. A divided polynomial is regular if it is not identically zero on any ideal $\til{L^m}=L^m\otimes_F E$, $m\geq 1$. We use Lemma \ref{Lemma8} to establish existence of a regular divided polynomial whose every value is divided ad-nilpotent of degree $k\geq 3$. Then we use Kostrikin-type arguments (\cite{K1}, \cite{K2}, \cite{Z7}) to reduce $k$ to 3.

In section \ref{J alg}, we show how such regular divided polynomials give rise to a family of quadratic Jordan algebras. This result is new only for $p=2$ or $3$. For $p\geq 5$, it follows from \cite{FGG}. Using structure theory of quadratic Jordan algebras \cite{MZ}, we establish existence of a regular Jordan polynomial, every value of which is an absolute zero divisor. These references are an essential (hidden) part of the proof. This Jordan polynomial gives rise to a regular divided polynomial whose every value is divided ad-nilpotent of degree 2, i.e., is a sandwich (see \cite{K2}, \cite{KZ}).

In sections \ref{Section5} and \ref{Section6}, we further push the envelope and construct a regular divided polynomial whose every value generates a nilpotent ideal in an associative enveloping algebra to reduce the problem to the case when the associative enveloping algebra $\til{A}$ satisfies a polynomial identity and finish the proof using structure theory of PI-algebras.

For an arbitrary (not necessarily associative) $F$-algebra $A$ and its Lie algebra of derivations $D=\text{Der}(A)$ denote $\til{A}=A\otimes_FE$, $\til{D}=D\otimes_FE$.  Clearly, $\til{D}\subseteq\text{Der}(\til{A})$. Let $i\in N$ and let $\til{D_i}=\sum\limits_\pi D\otimes e_\pi$, where the sum is taken over all ordered subsets of $\pi$ that contain $i$.  Clearly, $\til{D_i}\lhd\til{D}$, $\til{D_i}^2=(0)$, $(AD_i)(AD_i)=(0)$, and $\til{D}=\sum\limits_i\til{D_i}$.

Let $\Omega$ be a finite family of elements of $\til{D}$ such that
\begin{enumerate}[(U1)]
	\item \label{U1} every element $d\in\Omega$ lies in some ideal $\til{D_i}$,
	\item \label{U2} $[d_1,d_2]=0$ for arbitrary elements $d_1, d_2\in\Omega$,.
\end{enumerate}

Consider the following linear operator on $\til{A}$:
\[U_k(\Omega)=\sum d_1\cdots d_k,\]
where the summation runs over all $k$-element subsets of $\Omega$.

We further define $U_0(\Omega)=\text{Id}$.  Clearly, $U_1(\Omega)=\sum\limits_{d\in\Omega} d$.

We have $k! U_k(\Omega)=(\sum\limits_{d\in\Omega} d)^k$ and if the characteristic of the field exceeds $k$ then 
\[U_k(\Omega)=\frac{1}{k!}(\sum\limits_{d\in\Omega} d)^k.\]
Hence the operators $U_k$ play the role of divided powers.

The following properties of the operators $U_k(\Omega)$ are straightforward (see also \cite{Z5,Z8}).

\begin{lemma}
	\label{Lemma9} \hfill{}
	\begin{enumerate}
		\item $(ab)U_m(\Omega)=\sum\limits_{i=0}^m(aU_i(\Omega))(bU_{m-i}(\Omega))$ for arbitrary elements \linebreak$a,b\in\til{A}$, $m\geq0$;
			\item The operator $\til{A}\rightarrow\til{A}$, $a\rightarrow\sum\limits_{i=0}^\infty aU_i(\Omega)$ is an automorphism of the algebra $\til{A}$.  We remark that the sum $\sum\limits_{i=0}^\infty a U_i(\Omega)$ is finite;
			\item \label{Lemma9.3} For an element $a\in\til A$ let $R(a)$ denote the operator of right multiplication by $a$.  Then \[R(aU_m(\Omega))=\sum\limits_{i=0}^m(-1)^iU_i(\Omega)R(a)U_{m-i}(\Omega);\]
			\item $U_i(\Omega)U_j(\Omega)={i+j \choose i} U_{i+j}(\Omega)$.
		\end{enumerate}
		\end{lemma}

\bigskip

\begin{remark*}
This lemma will be primarily applied to Lie algebras where the operator $R(a)$ of right multiplication by an element $a$ is the adjoint operator $\ad(a)$.
\end{remark*}

An arbitrary element $a\in\til{A}$ can be uniquely represented as $a=\sum a_\pi$ where $a_\pi\in A\otimes e_\pi$.  We call it the standard decomposition of $a$.

Let $X$ be a countable set and let $I$ be an ideal of a Lie algebra $L$.  Consider the set Map($X,\til{I})$ of mappings $X\rightarrow\til{I}=I\otimes_FE$.  Consider also $M(I)=\text{Map(Map}(X,\til{I}),\til{I}).$  In other words, if $f\in M(I)$ and we assign values from $\til{I}$ to variables from $X$ then $f$ takes values in $\til{I}$.

Let Lie$\langle X\rangle$ be the free Lie algebra on the set of free generators $X$.  Consider the free product $L\ast \Lie\langle X \rangle$.  Let $(X)$ be the ideal of the algebra $L\ast Lie\langle X\rangle$ generated by $X$.  An arbitrary element from $(X)$ gives rise to an element from $M(I)$.

We will define a subset $U(I)\subseteq M(I)$ that we will call the set of divided polynomials defined on $I$:
\begin{enumerate}[(DP1)]
	\item \label{DP1} All elements from $(X)$ lie in $U(I)$;

\vspace{7mm}

	\item \label{DP2} suppose that a divided polynomial $w$ does not depend on any variables except $x_1, \cdots, x_r$. We represent this fact as $w=w(x_1, \cdots, x_r).$ If $v_1, \cdots, v_r \in U(I)$, then $w(v_1, \cdots, v_r) \in U(I)$ as well. \\
If $w=w(x_1, \cdots, x_r)$ and $v_i=v_i(y_1, \cdots, y_m)$, $1 \leq i \leq r$, are homogeneous divided polynomials of degrees $deg_{x_i}(w)$, $deg_{y_k}(v_i)$ in each variable, then $w(v_1, \cdots, v_r)$ is a homogeneous divided polynomial of degrees $\sum\limits_{i=1}^rdeg_{x_i}(w)\cdot deg_{y_k}(v_i)$ in $y_k$, $1\leq k\leq m$;

\vspace{7mm}

	\item \label{DP3} let $w=w(x_1,\cdots, x_r) \in U(I)$. Suppose that
	\begin{enumerate}[i)]
		\item \label{DP3i} for arbitrary elements $a, b, a_2, \cdots, a_r \in \til{I}$, we have 
			\[[w(a, a_2, \cdots, a_r), w(b, a_2, \cdots, a_r)]=0;\]
		\item \label{DP3ii} $w$ is linear in $x_1$, which means that $w(\alpha a + \beta b, a_2, \cdots, a_r) = \alpha w(a, a_2, \cdots, a_r) + \beta w(b, a_2, \cdots, a_r)$ for arbitrary $\alpha, \beta \in E$; $a, b, a_2, \cdots, a_r \in \til{I}$ and 
			\[w(I \otimes e_\pi, a_2, \cdots, a_r) \subseteq I \otimes e_\pi + I \otimes e_\pi E.\]
			Then for an arbitary $k \geq 0$ the function $w' = x_0 \text{ ad}_{x_1}^{[k]}(w),$ where $w'$ is defined as 
			\[w'(a_0, a_1, \cdots, a_r) = a_0U_k(\Omega ),\] 
			\[\Omega = \{\text{ad}(w(a_{1\pi},a_2, \cdots, a_r)\}_\pi,\]
			where $a_1=\sum\limits_\pi a_{1\pi}$ in the standard decomposition of the element $a_1$, is a divided polynomial defined on $I$.\\
	\end{enumerate}
\end{enumerate}

If $w$ is a homogeneous divided polynomial of degrees $deg_{x_i}(w)$ in $x_1, \cdots, x_r$, then $w'$ is a homogeneous divided polynomial of degrees $1, \kappa \cdot deg_{x_i}(w)$, $1 \leq i \leq r$, in $x_0, x_1, x_2, \cdots, x_r$.

An element of $M(I)$ lies in $U(I)$ if and only if starting with elements from $(X)$ and using rules DP\ref{DP2}-DP\ref{DP3}, it can be shown to be a divided polynomial.

 A divided polynomial from $U(I)$ is a homogeneous divded polynomial if and only if starting with homogeneous elements from $(X)$ and applying rules DP\ref{DP2}-DP\ref{DP3} to homogeneous polynomials, it can be shown to be a homogeneous divided polynomial.

Let's recall the definition of a polynomial map of vector spaces.

Let $V,W$ be vector spaces over an infinite field $F$ and let \\$f:\underbrace{V \times \cdots \times V}_{m} \rightarrow W$, $(v_1, \cdots, v_m) \rightarrow f(v_1, \cdots, v_m) \in W$. If $f$ is multilinear then it is said to be a polynomial of degrees $(1, 1, \cdots, 1)$ in $v_1, \cdots, v_m$. Let $d_i\geq 1$, $i=1, \cdots, m$. We say that $f$ is a homogeneous polynomial map of degrees $(d_1, \cdots, d_m)$ in $v_1, \cdots, v_m$ if
\begin{enumerate}[(1)]
\item for $d_i=1$ $f$ is linear in $v_i$;
\item for $d_i\geq 2$ we have
\[f(v_1,\cdots, v_{i-1}, v_i'+v_i'',v_{i+1},\cdots, v_m) -f(v_1, \cdots, v_{i-1}, v_i',v_{i+1}, \cdots, v_m)\]
\[-f(v_1,\cdots, v_{i-1},v_i'',v_{i+1},\cdots, v_m)\]
\[=\sum\limits_{k=1}^{d_i-1}f_k(v_1,\cdots,v_{i-1},v_i',v_i'',v_{i+1},\cdots,v_m),\]
where $f_k(v_1,\cdots, v_{i-1},v_i',v_i'',v_{i+1},\cdots,v_m)$ is a homogeneous polynomial map of degrees $d_1,\cdots,d_{i-1},k,d_i-k,d_{i+1},\cdots, d_m$ in \\$v_1,\cdots,v_{i-1},v_i',v_i'',v_{i+1},\cdots,v_m$.
\end{enumerate}

Now recall the definition of a full linearization of a homogeneous polynomial map $f$ of degrees $d_1\geq 1, \cdots, d_m\geq 1$ in $v_1,\cdots, v_m$. For every $1 \leq i \leq m$ choose $d_i$ elements $v_{i1}, \cdots, v_{id_i}\in V$. Let $\pi \subseteq \{v_{i1}, \cdots, v_{id_i}\}$ be a nonempty subset. Denote \\$f_\pi =f(v_1, \cdots, v_{i-1}, \sum\limits_{v \in \pi} v, v_{i+1}, \cdots, v_m).$ The mapping
\[\Delta_i(f) = \sum\limits_{\varnothing \neq \pi \subseteq\{v_{i1},\cdots, v_{id_i}\}} (-1)^{(d_i-|\pi|)} f_\pi\]
is called the linearization of $f$ with respect to $v_i$. The mapping \break $\Delta_i(f)(v_1, \cdots, v_{i-1}, v_{i1}, \cdots, v_{id_i}, v_{i+1}, \cdots, v_m)$ is multilinear in $v_{i1}, \cdots, v_{id_i}$.

Consecutively applying linearizations with respect to all variables, we get the full linearization $\til{f}: \underbrace{V \times \cdots \times V}_{\sum\limits_{i=1}^m d_i} \rightarrow W$. Clearly $\til{f}$ is a multilinear map.

\begin{lemma}
\label{Lemma33}
For an arbitrary homogeneous divided polynomial $w \in U(I)$
\begin{enumerate}[1)]
	\item the full linearization $\til{w}$ of $w$ lies in $L * \text{Lie}\langle X \rangle$;
	\item the span of all values of $w$ on $\til{I}$ is equal to the span of all values of $\til{w}$ on $\til{I}$.
\end{enumerate}
\end{lemma}
\begin{proof}
$1)$ We will use induction on the number of steps DP\ref{DP2}-DP\ref{DP3} needed to construct the divided polynomial $w$.

If $v_1, \cdots, v_r, w$ are homogeneous divided polynomials, then the full linearization of $w(v_1, \cdots, v_r)$ is a linear combination of values $\til{w}(\til{v_1}, \cdots, \til{v_r})$ in appropriate variables. 

Let $w=x_0\text{ ad}_{x_1}^{[k]}(v)$, where $v=v(x_1, \cdots, x_r)$ in a homogeneous divided polynomial satisfying the conditions DP3 (i), (ii). Since we can linearize variables in an arbitrary order, let's start with the variable $x_1$. Then
\[\Delta_{x_1}(w)(x_0, y_1, \cdots, y_k, x_2, \cdots, x_r) \]
\[= x_0\text{ ad}(v(y_1, x_2, \cdots, x_r)) \cdots \text{ ad}(v(y_k, x_2, \cdots, x_r)).\]

This completes the proof of assertion 1).

We will prove part 2) of the Lemma in a slightly more general context of polynomial maps of spaces. Consider again vector spaces $V, W$ and a homogeneous polynomial map $f: \underbrace{V \times \cdots \times V}_m \rightarrow W$, $(x_1, \cdots, x_m) \rightarrow f(x_1, \cdots, x_m) \in W$, $x_i \in V$. Let $f$ have degrees $d_1 \geq 1, \cdots, d_m \geq 1$ with respect to $x_1, \cdots, x_m$. Choose $x_i', x_i'' \in V$, $1 \leq i \leq m$. Consider
\[f(x_1, \cdots, x_{i-1}, x_i'+x_i'', x_{i+1}, \cdots, x_m) - f(x_1, \cdots, x_{i-1}, x_i', x_{i+1}, \cdots, x_m) \]
\[- f(x_1, \cdots, x_{i-1}, x_i'', x_{i+1}, \cdots, x_m) \]
\[= \sum\limits_{j=1}^{d_1-1} f_j(x_1, \cdots, x_{i-1}, x_i', x_i'', x_{i+1}, \cdots, x_m),\]
where the summand $f_j$ has degree $j$ in $x_i'$ and degree $d_i-j$ in $x_i''$. The homogeneous polynomial mappings $f_j$ are called partial linearizations of $f$.

Consider the finite system $\mathcal{F}$ of homogeneous polynomial maps from $V$ to $W$ that are obtained from $f$ by repeated partial linearizations.

Let $\Omega \subset V$ be a family of elements with the following property: \\
\[ \text{if }g(x_1, \cdots, x_{r}) \in \mathcal{F} \text{ and } g \text{ has degree} \geq 2 \text{ in } x_i, \text{ then for an }\label{star} \tag{$\ast$}\]
\hspace{.74cm} arbitrary element  $v\in \Omega$ we have
\[ g(\underbrace{V, \cdots, V}_{i-1}, v, \underbrace{V, \cdots, V}_{r-i}) = (0).\]

We claim that for an arbitrary element $g(x_1, \cdots, x_r) \in \mathcal{F}$, 
\[g(\text{span } \Omega, \cdots, \text{span } \Omega) \subseteq \text{ span}\til{f}(\Omega, \cdots, \Omega). \]

Applying this inclusion to $f=w$, $\Omega=\bigcup\limits_{i\geq1}(I\otimes e_i+I\otimes e_iE)$, we will prove part $2)$ of Lemma \ref{Lemma33}.

If $g$ is multilinear, then $g = \til{f}$. In any case, without loss of generality, we assume that the claim is true for all partial linearizations of $g$. But modulo partial linearizations the mapping $g$ is multilinear. More precisely,
\[g(\text{span } \Omega, \cdots, \text{span } \Omega) \subseteq\]
\[ \text{ span }(g(v_1, \cdots, v_r), v_i \in \Omega) + \sum g'( \text{span } \Omega, \cdots, \text{span }\Omega),\]
where $g'$ are partial linearizations of $g$. Since $g$ has degree $\geq 2$ with respect to at least one variable, we conclude that $g(v_1, \cdots, v_r) = 0.$

In particular,
\[\text{span }f(\text{span }\Omega, \cdots, \text{span }\Omega)= \text{span }(\til{f}(v_1, v_2, \cdots), v_i \in \Omega).\]
If $f$ is a homogeneous divided polynomial and $\Omega =\{a \otimes e_i, a\in L, i\geq 1\}$, then condition (\ref{star}) is clearly satisfied, which completes the proof of assertion 2).
\end{proof}

The following lemma is a linearization version of the celebrated Kostrikin Lemma ([Kos1; Kos2, Lemma 2.1.1]).

\begin{lemma}
\label{Lemma34}
Let $L$ be a Lie algebra. Let $\Omega \subset \text{Der }(L)\otimes E$ be a finite family of elements satisfying the conditions (U\ref{U1}), (U\ref{U2}). Suppose that $m \geq 1$ and for an arbitary $k \geq m$, we have $U_k(\Omega) = 0$.
\begin{enumerate}[1)]
	\item Let $m\geq 2$. Then for arbitary elements $a, b \in L$, we have
		\[[aU_{m-1}(\Omega), bU_{m-1}(\Omega)] =0.\]
	\item Now suppose that $m \geq 4$. Let $a \in \til{L},$ $a = \sum\limits_\pi a_\pi$ be a standard decomposition, and $\Omega' = \{a_\pi U_{m-1}(\Omega)\}_\pi$. Then $U_k(\Omega')=0$ for $k \geq m-1.$
\end{enumerate}
\end{lemma}
\begin{proof}
1) We have $2m-2 \geq m$. Adjoint operators are right multiplications in Lie algebras. Hence Lemma \ref{Lemma9} $(3)$ is applicable. By Lemma \ref{Lemma9} (3),
\[0=\text{ad}(bU_{2m-2}(\Omega)) = \sum\limits_{i+j=2m-2}\pm U_i(\Omega) \text{ ad}(b)U_j(\Omega).\]
It implies
\[U_{m-1}(\Omega)\text{ ad}(b)U_{m-1}(\Omega)= \sum\limits_{\substack{i\geq m \\ \text{or } j \geq m}} \pm U_i(\Omega)\text{ ad}(b)U_j(\Omega)=0.\]
By Lemma \ref{Lemma9} $(3)$, $(4)$, we have
\begin{align*}
[aU_{m-1}(\Omega), bU_{m-1}(\Omega)]&=aU_{m-1}(\Omega)\ad\left(bU_{m-1}(\Omega)\right)\\
&=aU_{m-1}(\Omega)\ad(b)U_{m-1}(\Omega)\\
&=0;
\end{align*}
which completes the proof of the assertion 1).

2) We will show that 
\[\text{ad}(a_1U_{m-1}(\Omega))\cdots\text{ad}(a_kU_{m-1}(\Omega))=0\]
for arbitrary elements $a_1, \cdots, a_k \in L$, $k \geq m-1$. Without loss of generality we will assume $k=m-1$.

By Lemma \ref{Lemma9} (3) and (4) the left hand side is a linear combination of generators
\[U_{i_0}(\Omega)\text{ ad}(a_1)U_{i_1}(\Omega) \cdots U_{i_{m-2}}(\Omega)\text{ ad}(a_{m-1})U_{i_{m-1}}(\Omega),\]
where $0 \leq i_0, i_1, \cdots, i_{m-1} \leq m-1$ and $i_0+i_1+\cdots+i_{m-1}=(m-1)^2.$

Suppose that $U_{i_0}(\Omega)\text{ ad}(a_1)\cdots\text{ad}(a_{m-1})U_{i_{m-1}}(\Omega)\neq 0$ and the $m$-tuple $(i_0, i_1, \cdots, i_{m-1})$ is lexicographically maximal with this property.

We claim that none of the indices $i_0, i_1, \cdots, i_{m-1}$ are equal to 0. Indeed, if one of the indices is equal to 0, then all other indices have to be equal to $m-1$. Since $m \geq 4$ it follows that there exists $t$, $0 \leq t \leq m-1$, such that $i_t = i_{t+1}=m-1$. Now from 1) it follows that $U_{i_t}(\Omega) \text{ ad}(a_{t+1})U_{i_{t+1}}(\Omega) = U_{m-1}(\Omega) \text{ ad}(a_{t+1})U_{m-1}(\Omega)=0$, a contradiction.

Since $(m-2)m< (m-1)^2$ it follows that at least one index $i_t$, $0 \leq t \leq m-1$, is equal to $m-1$. All of the indices $i_1, \cdots, i_{m-1}$ are smaller than $m-1$. Indeed, suppose that $i_t=m-1$, $1 \leq t\leq m-1$. We have $i_{t-1} \geq 1$ by the above. Now Lemma \ref{Lemma9} (3) implies
\[0= \text{ad}(a_tU_{i_{t-1}+i_t}(\Omega))=\sum\limits_{i+j=i_{t-1}+i_t} \pm U_i(\Omega) \text{ ad}(a_t)U_j(\Omega)\]
and therefore
\[U_{i_{t-1}}(\Omega)\text{ ad}(a_t)U_{i_t}(\Omega)=\sum\limits_{\substack{i> i_t \\ j < m-1}} \pm U_i(\Omega) \text{ ad}(a_t)U_j(\Omega),\]
which contradicts lexicographical maximality of $(i_0, \cdots, i_{m-1})$.

We have proved that $i_0=m-1,$ $i_1=i_2=\cdots = i_{m-1}=m-2$. Now our aim will be to show that
\[U_{m-1}(\Omega)\text{ ad}(a_1)U_{m-2}(\Omega)=0.\]
Since $(m-1)+(m-3) \geq m$, Lemma \ref{Lemma9} (3) implies that
\[U_{m-1}(\Omega)\text{ ad}(a_1)U_{m-3}(\Omega) - U_{m-2}(\Omega)\text{ ad}(a_1)U_{m-2}(\Omega) \]
\[+ U_{m-3}(\Omega)\text{ ad}(a_1)U_{m-1}(\Omega)=0.\]
Multiplying the left hand side by $U_1(\Omega)$ on the right and taking into account Lemma \ref{Lemma9} (4), we get
\[(m-2)U_{m-1}(\Omega)\text{ ad}(a_1)U_{m-2}(\Omega) -(m-1)U_{m-2}(\Omega)\text{ ad}(a_1)U_{m-1}(\Omega) = 0.\]
On the other hand Lemma \ref{Lemma9} (3) implies that
\[U_{m-1}(\Omega)\text{ ad}(a_1)U_{m-2}(\Omega) - U_{m-2}(\Omega)\text{ ad}(a_1)U_{m-1}(\Omega)=0.\]
The system of equations implies
\[U_{m-1}(\Omega)\text{ ad}(a_1)U_{m-2}(\Omega)=0,\]
which completes the proof of the lemma.
\end{proof}

\begin{definition}
	We say that a divided polynomial $w(x_1,\cdots,x_r)$ defined on $L^s,\ s\geq1$, is \underline{regular} if for an arbitrary $i\geq s$ we have $w(\til{L}^i,\cdots,\til{L}^i)\neq(0)$.

\end{definition}

By Lemma \ref{Lemma33} $(2)$, a homogeneous divided polynomial $w$ is regular if and only if its full linearization is regular.

\begin{lemma}
	\label{Lemma10}There exist integers $m\geq1,N\geq1$ and a regular homogeneous divided polynomial $w(x_1,\cdots,x_r)$ defined on $L^m$ such that $w$ satisfies the conditions in (DP3) and $x_0\ad_{x_1}^{[t]}(w)=0$ holds identically on $\til L^m$ for all $t \geq N$.
	\begin{proof}
		Consider the elements $c_1,\cdots,c_r\in S$ and integers $m\geq1,\ N\geq1$ of \Cref{Lemma8}.  By property 2), for an $i\geq m$ the subspace $L_i=[L^i,c_1,\cdots,c_r]$ is a subalgebra of $L$.  By 3) this subalgebra is nilpotent, say, of degree $d(i),\ d(m)\geq d(m+1)\geq\cdots$.  This sequence stabilizes at some step, $d=d(k)=d(k+1)=\cdots$.  Thus $L_k^d=(0)$ and $L_i^{d-1}\neq(0)$ for any $i\geq m$.  Let 
		\begin{multline*}w(x_1,\cdots,x_{d-1})=[[x_1,c_1,\cdots,c_r],[x_2,c_1,\cdots,c_r],\cdots\\
		\cdots,[x_{d-1},c_1,\cdots,c_r]]\in L\ast\Lie\langle X\rangle.
		\end{multline*}
		The divided polynomial $w$ is regular and linear in $x_1$.  For arbitrary elements $a_2,\cdots,a_{d-1}\in L^k$ we have \[[w(L^k,a_2,\cdots,a_{d-1}),w(L^k,a_2,\cdots,a_{d-1})]=(0)\]  
because the left hand side lies in $L_k^d$. Hence the divided polynomial $w$ satisfies the condition (DP3). Therefore the divided polynomial $x_0 \text{ad}_{x_1}^{[t]}(w)$ is defined on $L^k$ for any $t \geq 1.$ For $t \geq N$ the polynomial $x_0\text{ad}_{x_1}^{[t]}(w)$ is identically zero on $\til{L}^k$ by Lemma \ref{Lemma8} (3). This finishes the proof of the lemma.
	\end{proof}
\end{lemma}

Let $q\geq1$ be a minimal integer with the following property:

there exists an $m \geq 1$ and a regular homogeneous divided polynomial $w=w(x_1,\cdots,x_r)$ defined on $L^m$, linear in $x_1$, such that 
\begin{enumerate}[i)]
	\item for arbitrary elements $a,b,a_2,\cdots,a_r\in \til L^m$ we have \[[w(a,a_2,\cdots,a_r),w(b,a_2,\cdots,a_r)]=0;\]
	\item $\til L^m\ad_{x_1}^{[t]}(w)=(0)$ holds identically on $\til L^m$ for all $t \geq q.$ Clearly, $q \leq N.$
\end{enumerate}

\begin{lemma}
	\label{Lemma11}$q\leq3$.
	\begin{proof}
		Suppose that $q\geq4$.  Consider the divided polynomial \linebreak $v(x_0,x_1,\cdots,x_r)=x_0\ad_{x_1}^{[q-1]}(w)$ defined on $L^m$.  In view of the minimality of $q$, the divided polynomial $v$ is regular.
		
		By Lemma \ref{Lemma34} (1) for arbitrary elements $a,b\in\til L^m;\ a_1,\cdots,a_r\in\til L^m$ we have
		\[[v(a,a_1,\cdots,a_r),v(b,a_1,\cdots,a_r)]=0.\]
		
		We proved that the divided polynomial $y\ad_{x_0}^{[q-1]}(v(x_0,\cdots,x_r))$ is defined on $\til L^m$.  If $q\geq4$, then by Lemma \ref{Lemma34} (2), this divided polynomial is identically zero, which contradicts the minimality of $q$ and finishes the proof of the lemma.
	\end{proof}
\end{lemma}

\begin{lemma}
	\label{Lemma12}
	Let $L$ be a Lie algebra. Let $\Omega \subset \til{L}$ be a finite family of elements such that ad($\Omega$) satisfies the assumptions (U\ref{U1}), (U\ref{U2}). Suppose that $U_2(\text{ad}(\Omega))=0$. Then $a = \sum\limits_{b \in \Omega} b$ is a sandwich of the Lie algebra $\til{L}$.
	\begin{proof}
		We have $\text{ad}(a)^2=U_1(\text{ad}(\Omega))U_1(\text{ad}(\Omega))=2U_2(\text{ad}(\Omega))=0.$ By Lemma \ref{Lemma9} (\ref{Lemma9.3}) for an arbitrary element $c \in \til{L}$ we have 
		\begin{multline*}
			\text{ad}(cU_2(\text{ad}(\Omega)))=\text{ad}(c)U_2(\text{ad}(\Omega))-U_1(\text{ad}(\Omega))\text{ad}(c)U_1(\text{ad}(\Omega))\\
			+ U_2(\text{ad}(\Omega))\text{ad}(c),
		\end{multline*}
		which implies $\text{ad}(b)\text{ad}(c)\text{ad}(b) = 0$ and completes the proof of the lemma.
	\end{proof}
\end{lemma}


In what follows, we will use the subsequent lemma.

\begin{lemma}
\label{Lemma40}
Let $L$ be a Lie algebra. Let $\Omega=\{a_1,\cdots, a_n\}\subset L$ be a finite family of elements. Let $\Omega=\Omega_1\dot{\cup}\cdots\dot{\cup}\Omega_s=\Omega'_1\dot{\cup}\cdots\dot{\cup}\Omega'_t$ be two disjoint decompositions. Denote $b_k=\sum\limits_{a_i\in\Omega_k}a_i$, $c_\ell=\sum\limits_{a_j\in\Omega'_\ell}a_j$. Denote also $\ad[\Omega,\Omega]=\spn(\ad[a_i,a_j],1\leq i,j\leq n)$.

Suppose that if $a_i,a_j$ lie in the same $\Omega_k$ or in the same $\Omega'_\ell$, then $\ad(a_i)\ad(a_j)=0$. Then $\sum\ad(b_{k_1})\ad(b_{k_2})=\sum\ad(c_{\ell_1})\ad(c_{\ell_2})$ mod $\ad[\Omega,\Omega]$, where both sums run over all 2-element subsets $\{k_1,k_2\}\subseteq\{1,\cdots,s\}$ and $\{\ell_1,\ell_2\}\subseteq\{1,2,\cdots,t\}$ respectively.
\end{lemma}
\begin{proof}
It is easy to see that
\begin{align*}
\sum\ad(b_{k_1})\ad(b_{k_2})&=\sum\ad(c_{\ell_1})\ad(c_{\ell_2})\\
&=\sum\limits_{1\leq i<j\leq n}\ad(a_i)\ad(a_j) \text{ \hspace{0.1cm} mod} \ad[\Omega,\Omega].
\end{align*}
\end{proof}

\section{Jordan algebras}
\label{J alg}

Let $J$ be a vector space over a field $F$ (of arbitrary characteristic) with two quadratic mappings $J\rightarrow J,\ x\rightarrow x^2$, and $Q:J\rightarrow End_F(J)$.  For elements $x,y,z\in J$ denote
\[x\circ y=(x+y)^2-x^2-y^2,\ \{x,y,z\}=y(Q(x+z)-Q(x)-Q(z)).\]

Following K. McCrimmon (\cite{M1}) we say that $(J,x\rightarrow x^2,Q)$ is a quadratic Jordan algebra if it satisfies the identities
\begin{enumerate}[(M1)]
	\item $\{x,x,y\}=x^2\circ y$;
	\item $(yQ(x))\circ x=(y\circ x)Q(x)$;
	\item $x^2Q(x)=(x^2)^2$;
	\item $x^2Q(y)Q(x)=(yQ(x))^2$;
	\item $Q(x^2)=Q(x)^2$;
	\item $Q(yQ(x))=Q(x)Q(y)Q(x)$
	\vspace{-.4 cm}
\end{enumerate}
and all their partial linearizations.

We reiterate the assumption made at the beginning of section \ref{Section3}: all algebras are considered over an infinite field $F$ of characteristic $p>0.$

Let $w'=w'(x_1,\cdots,x_{r-1})$ be a regular homogeneous divided polynomial defined on $L^m$ such that $w'$ is linear in $x_1$ and satisfies all the assumptions of (DP\ref{DP3}). Moreover, assume that $\til{L^m}\text{ ad}_{x_1}^{[k]}(w')=0$ holds identically for $k\geq 3$. If there exists $s \geq m$ such that $x_r\text{ ad}_{x_1}^{[2]}(w')$ holds identically on $\til{L^s}$ then our goal of constructing a sandwich valued regular homogeneous divided polynomial has been achieved. We assume therefore that the divided polynomial $w(x_1,\cdots,x_r)=x_r\ad_{x_1}^{[2]}(w')$ is regular.

The divided polynomial $w$ satisfies both assumptions of (DP3): it is clearly linear in $x_r$ and for arbitrary elements $a,b,a_1,\cdots, a_{r-1}\in\til{L^m}$, we have
\[[a\ad_{x_1}^{[2]}w'(a_1,\cdots,a_{r-1}), b\ad_{x_1}^{[2]}w'(a_1,\cdots, a_{r-1})]=0\]
by Lemma \ref{Lemma34} (1).

Choose $a_1,\cdots, a_r \in \til{L^m}$ and denote $a'=w'(a_1,\cdots, a_{r-1})$, \linebreak$a=w(a_1,\cdots,a_r)$. Denote $\ad^{[k]}(a')=(\ad_{x_1}^{[k]}w')(a_1,\cdots,a_{r-1})$. If \linebreak $a_r=\sum\limits_\pi a_{r\pi}$ is the standard decomposition, then we denote $\ad^{[k]}(a)=(\ad_{x_r}^{[k]}w)(a_1,\cdots,a_r)=\sum\ad(a_{r\pi_1}\ad^{[2]}(a'))\cdots\ad(a_{r\pi_k}\ad^{[2]}(a'))$, where the sum runs over all $k$-element sets $(\pi_1,\cdots,\pi_k)$.

Notice that $\ad^{[k]}(a)=0$ for $k\geq3$. Indeed, by Lemma \ref{Lemma9} (3), the equalities $\ad(a_{r\pi_i}\ad^{[3]}(a'))=0$ and $\ad(a_{r\pi_i}\ad^{[4]}(a'))=0$ imply

\[\label{eqn:ast}\tag{$\ast$} \ad^{[2]}(a')\ad(a_{r\pi_i})\ad(a')=\ad(a')\ad(a_{r\pi_i})\ad^{[2]}(a'),\]

\[\label{eqn:dast}\tag{$\ast\ast$}\ad^{[2]}(a')\ad(a_{r\pi_i})\ad^{[2]}(a')=0\]

\noindent respectively.

We have

\[\ad^{[k]}(a)=\sum\pm\ad^{[i_1]}(a')\ad(a_{r\pi_1})\ad^{[i_2]}(a')\cdots\ad(a_{r\pi_k})\ad^{[i_{k+1}]}(a'),\]

\noindent where $0\leq i_1, \cdots, i_{k+1}\leq 2$, $i_1+\cdots+i_{k+1}=2k$. If at least one $i_\mu$, $1\leq \mu\leq k+1$, is equal to 0, then all other $i_\nu$, $\nu\neq\mu$, are equal to 2. In this case, the product is equal to 0 by (\ref{eqn:dast}). Suppose that all $i_\mu\neq0$. Then all $i_\mu$, except two, are equal to 2. These two are equal to 1. Since $k+1\geq 4$, we have at least two degrees $i_\mu$ that are equal to 2. Using (\ref{eqn:ast}), we can move two operators $\ad^{[2]}(a')$ together and then use (\ref{eqn:dast}).

Consider the subspaces $K'_a=\{x\in\til{L}^m|x\ad^{[2]}(a)=0\}$ and $K_a=\sum\limits_i(L^m\otimes e_i+\til{L}^me_i)\cap K'_a$ and the factor space $J_a=\til{L}^m/K_a$.

Let $x=\sum\limits_{\pi}x_\pi$ be the standard decomposition of an element $x\in \til L^m$.  Define $x^2=a\sum\ad(x_{\pi_1})\ad(x_{\pi_2})+K_a$, where the sum runs over all 2-element sets $(\pi_1,\pi_2)$.  The order of the factors in $\ad(x_{\pi_1})\ad(x_{\pi_2})$ is irrelevant since $[a,\til L^m]\subseteq K_a$.  Define further
\[yQ(x)=y\ad^{[2]}(a)\sum\ad(x_{\pi_1})\ad(x_{\pi_2})+K_a\]

Again, the order of factors in $\ad(x_{\pi_1})\ad(x_{\pi_2})$ is irrelevant since \linebreak $y\ad^{[2]}(a)\ad(\til L^m)\subseteq K_a$.

Linearizing the above operations, we get $x\circ y =[[a,x],y]+K_a$ for $x,y\in J_a$, and $\{x,y,z\}=[y\ad^{[2]}(a),x,z]+K_a$ for $x,y,z\in J$.

\begin{lemma}
	\label{Lemma13} \hfill{}
	\begin{enumerate}
		\item the element $u=y_1\ad^{[i_1]}(a)\ad(y_2)\ad^{[i_2]}(a)\cdots\ad(y_s)\ad^{[i_s]}(a),$\linebreak where $y_1,\cdots,y_s\in\til L^m;\ i_1+\cdots+i_s\geq s+2$, is equal to 0;
		\item an operator $ad^{[i_1]}(a)\ad(y_1)\cdots\ad(y_s)\ad^{[i_{s+1}]}(a),$ for $y_1,\cdots,y_s\in\til L^m;\ i_1+\cdots,i_{s+1}\geq s+3,$ is equal to 0 on $\til L^m$;
		\item consider an operator 
\[v=\ad^{[i_1]}(a)\ad(y_1)\cdots\ad(y_s)\ad^{[i_{s+1}]}(a),\] 
where $y_1,\cdots,y_s\in \til L^m,\ i_1+\cdots+i_{s+1}\geq s+2$.  Suppose that there exists $1\leq k\leq s-2$ such that $i_{k+1}=i_{k+2}=0$, in other words $v=\sum\cdots\ad(y_k)\ad(y_{k+1})\ad(y_{k+2})\cdots$.  Then $v$ is zero on $\til L^m$.
	\end{enumerate}
	\begin{proof}
		To prove (1) we will use induction on $s$.  If $s=1$ then $u=y_1\ad^{[i_1]}(a),\ i_1\geq3$.  Hence $u=0.$  
		
		Let $s\geq2$.  If $i_1\geq3$ then again $u=0.$  If $i_1\leq1$ then choosing $y_1'=y_1\ad^{[i_1]}(a)\ad(y_2)$ we can use the induction assumption.  Therefore we let $i_1=2$.  If $i_2=0$ then choosing $y_1'=i_1\ad^{[2]}(a)\ad(y_2)\ad(y_3)$ we again use the induction assumption.  Let $i_2=1$.  Then by \linebreak\Cref{Lemma9}(3) we have
		\[\ad^{[2]}(a)\ad(y_2)\ad(a)=\ad(a)\ad(y_2)\ad^{[2]}(a),\]
		the case that has already been considered.  Finally, if $i_2=2$ then $\ad^{[2]}(a)\ad(y_2)\ad^{[2]}(a)=0,$ again by \Cref{Lemma9}(3), which finishes the proof of part (1).
		
		To prove (2) we consider the element
		\[y_0\ad^{[i_1]}(a)\ad(y_1)\cdots\ad(y_s)\ad^{[i_{s+1}]}(a)\]
		and use part (1).
		
		Consider now an operator $v=\ad^{[i_1]}(a)\ad(y_1)\cdots\ad(y_s)\ad^{[i_{s+1}]}(a)$ and suppose that \[v=v'\ad(y_k)\ad(y_{k+1})\ad(y_{k+2})v'',\] where
		 \[v'=\ad^{[i_{1}]}(a)\ad(y_1)\cdots\ad(y_{k-1})\ad^{[i_k]}(a),\] 

\[v''=\ad^{[i_{k+3}]}(a)\cdots\ad^{[i_{s+1}]}(a).\]

		 By part (2) if $v\neq0$ on $\til L^m$ then $i_1+\cdots+i_k\leq(k-1)+2=k+1,\ i_{k+1}+\cdots+i_{s+1}\leq(s-k-2)+2=s-k$.  However, $i_1+\cdots+i_k+i_{k+3}+\cdots+i_{s+1}\geq s+2$, a contradiction that finishes the proof of the lemma.	
	\end{proof}
\end{lemma}

\begin{lemma}
	\label{Lemma14} Let $\Omega$ be a finite family of commuting elements from $\til{L}$ such that every element from $\Omega$ lies in one of the ideals $L \otimes e_\pi + \til{L} e_\pi$. Denote for brevity $U_k(\ad(\Omega))=\ad^{[k]}(\Omega)$ and suppose that $\ad^{[3]}(\Omega)=\ad^{[4]}(\Omega)=0$. Then for arbitrary elements $y_1, y_2 \in \til{L}$ we have
	\[\ad(y_1\adx[2][](\Omega))\ad(y_2\adx[2][](\Omega))=\adx[2][](\Omega)\ad(y_1)\ad(y_2)\adx[2][](\Omega).\]
	\begin{proof}
		By \Cref{Lemma9}(3) we have
		\[\begin{multlined}\ad(y_i\adx[2][](\Omega))=\ad(y_i)\adx[2][](\Omega)-\\\adx[1][](\Omega)\ad(y_i)\adx[1][](\Omega)+\adx[2][](\Omega)\ad(y_i),\end{multlined}\]
		$i=1,2$.  By \cref{Lemma9}(4) we have also
		\[
		\adx[1][](\Omega)\adx[1][](\Omega)=2\adx[2][](\Omega),\]
		\[ \adx[1][](\Omega)\adx[2][](\Omega)=
		\adx[2][](\Omega)\adx[1][](\Omega)=3\adx[3][](\Omega)=0.\]
		
		Again by \Cref{Lemma9}(3) we have
		\begin{multline*}
		\ad(y_i\adx[3][](\Omega))=\ad(y_i)\adx[3][](\Omega)-\adx[1][](\Omega)\ad(y_i)\adx[2][](\Omega)+\\
		\adx[2][](\Omega)\ad(y_i)\adx[1][](\Omega)-\adx[3][](\Omega)\ad(y_i)=0,
		\end{multline*}
		which implies
		\[\adx[1][](\Omega)\ad(y_i)\adx[2][](\Omega)=\adx[2][](\Omega)\ad(y_i)\adx[1][](\Omega).\]
		Similarly, $\ad(y_i\adx[4][](\Omega))=0$ implies $\adx[2][](\Omega)\ad(y_1)\adx[2][](\Omega)=0.$  Hence,
		\begin{align*}
		&\ad(y_1\adx[2][](\Omega))\ad(y_2\adx[2][](\Omega))=
		(\ad(y_i)\adx[2][](\Omega)-\\
		&\adx[1][](\Omega)\ad(y_1)\adx[1][](\Omega)+	\adx[2][](\Omega)\ad(y_1))(\ad(y_2)\adx[2][](\Omega)-\\
		&\adx[1][](\Omega)\ad(y_2)\adx[1][](\Omega)+\adx[2][](\Omega)\ad(y_2))=\\
		&-\adx[1][](\Omega)\ad(y_1)\adx[1][](\Omega)\ad(y_2)\adx[2][](\Omega)+\\
		& 2\adx[1][](\Omega)\ad(y_1)\adx[2][](\Omega)\ad(y_2)\adx[1][](\Omega)+\\
		&\adx[2][](\Omega)\ad(y_1)\ad(y_2)\adx[2][](\Omega)-\\
		&\adx[2][](\Omega)\ad(y_1)\adx[1][](\Omega)\ad(y_2)\adx[1][](\Omega)=\\
		&\adx[2][](\Omega)\ad(y_1)\ad(y_2)\adx[2][](\Omega),
		\end{align*}
		which proves the lemma.
	\end{proof}
\end{lemma}

\begin{lemma}
\label{Lemma41}
\begin{enumerate}[(1)]
\item The operations $x\rightarrow x^2$ and $Q$ are well defined on $J_a$;
\item let $f:\til{L}^m\times\cdots\times\til{L}^m\rightarrow\til{L}^m$ be a homogeneous polynomial map, and let $\til{f}(x_1,\cdots, x_n)$ be its full linearization. Suppose that $\til{f}(L^m\otimes e_i+\til{L}^me_i,\til{L}^m,\cdots,\til{L}^m)\subseteq L^m\otimes e_i+\til{L}^me_i$ for all $i$.
\end{enumerate}
Then, if an arbitrary value of $f$ lies in $K'_a$, then an arbitrary value of $f$ lies in $K_a$.
\end{lemma}
\begin{proof}
(1) Choose arbitrary elements $x,y\in\til{L}^m$ and $z',z\in K_a$. We need to show that $(y+z')Q(x+z)=yQ(x)$ and $(x+z)^2=x^2$. Let $x=\sum\limits_\pi x_\pi$ be the standard decomposition of the element $x$. We have $z'Q(x)=z'\ad^{[2]}(a)\sum\ad(x_{\pi_1})\ad(x_{\pi_2})+K_a=0$ since $z'\in K_a\subseteq K'_a$. Hence $(y+z')Q(x+z)=yQ(x+z)$. Furthermore, it is easy to see that
\[yQ(x+z)=yQ(x)+y\ad^{[2]}(a)\ad(x)\ad(z)+yQ(z) \text{ \hspace{0.1cm} mod } K_a.\]

By Lemma \ref{Lemma14}, for an arbitrary standard component $x_\pi$ of the element $x$, we have
\[\ad^{[2]}(a)\ad(x_\pi)\ad(z)\ad^{[2]}(a)=\ad(x_\pi\ad^{[2]}(a))\ad(z\ad^{[2]}(a))=0,\]
since $z\in K'_a$. Hence $y\ad^{[2]}(a)\ad(x_\pi)\ad(z)\in K_a$ and \linebreak $y\ad^{[2]}(a)\ad(x)\ad(z)\in K_a$.

Let us show that $yQ(z)=0$. We have $z=z_1+\cdots+z_s$, where $z_i\in(L^m\otimes e_i+\til{L}^me_i)\cap K'_a$. Let $z_i=\sum\limits_\pi z_{i\pi}$ be the standard decomposition of the element $z_i$. Then $z=\sum\limits_\pi z_\pi$, $z_\pi=\sum\limits_i z_{i\pi}$, is the standard decomposition of the element $z$. Consider the family of elements $\Omega=\{z_{i\pi}\}_{i,\pi}$ and two decompositions $\Omega=\bigcup\Omega_i$, $\Omega_i=\{z_{i\pi}\}_\pi$, and $\Omega=\bigcup\limits_\pi\Omega'_\pi$, $\Omega'_\pi=\{z_{i\pi}\}_i$. By Lemma \ref{Lemma40}, we have 
\begin{align*}
y\ad^{[2]}(a)&\sum\ad(z_{\pi_1})\ad(z_{\pi_2})\\
&=y\ad^{[2]}(a)\sum\limits_{1\leq i<j\leq s}\ad(z_i)\ad(z_j)\text{ \hspace{0.1cm} mod } y\ad^{[2]}(a)\ad(\til{L}^m).
\end{align*}
Recall that $y\ad^{[2]}(a)\ad(\til{L}^m)\subseteq K_a$. The element $y\ad^{[2]}(a)\ad(z_i)\ad(z_j)$ lies in $L^m\otimes e_i+\til{L}^me_i$ and 
\[y\ad^{[2]}(a)\ad(z_i)\ad(z_j)\ad^{[2]}(a)=y\ad(z_i\ad^{[2]}(a))\ad(z_j\ad^{[2]}(a))=0\]
by Lemma \ref{Lemma14}. Hence, $y\ad^{[2]}(a)\ad(z_i)\ad(z_j)\in K_a$. This implies $yQ(z)=0$.

Now let us show that $(x+z)^2=x^2$. We have $(x+z)^2=x^2+a\ad(z)\ad(x)+z^2$ mod $K_a$. For an arbitrary standard component $x_\pi$,
\begin{align*}
a\ad(z)\ad(x_\pi)\ad^{[2]}(a)&=-z\ad(a)\ad(x_\pi)\ad^{[2]}(a)\\
&=-z\ad^{[2]}(a)\ad(x_\pi)\ad(a)\\
&=0
\end{align*}
by Lemma \ref{Lemma9} (3). Hence $a\ad(z)\ad(x_\pi)\in K_a$ and $a\ad(z)\ad(x)\in K_a$.

Let us show that $z^2=0$. We have $z^2=a\sum\ad(z_{\pi_1})\ad(z_{\pi_2})+K_a$. By Lemma \ref{Lemma40},
\[a\sum\ad(z_{\pi_1})\ad(z_{\pi_2})=a\sum\limits_{1\leq i<j\leq s}\ad(z_i)\ad(z_j) \text{ \hspace{0.1cm} mod } a\ad(\til{L}^m)\subseteq K_a.\]

As above, 
\begin{align*}
a\ad(z_i)\ad(z_j)\ad^{[2]}(a)&=-z_i\ad(a)\ad(z_j)\ad^{[2]}(a)\\
&=z_i\ad^{[2]}(a)\ad(z_j)\ad(a)\\
&=0
\end{align*}
by Lemma \ref{Lemma9} (3). Hence $a\ad(z_i)\ad(z_j)\in K_a$ and \linebreak $a\sum\ad(z_{\pi_1})\ad(z_{\pi_2})\in K_a$. This completes the proof of part (1).
(2) Now let $f:\til{L}^m\times\cdots\times\til{L}^m\rightarrow\til{L}^m$ be a homogeneous polynomial map with the full linearization $\til{f}$. By Lemma \ref{Lemma33} (2), for polynomial maps, the $F$-linear span of all values of $f$ is equal to the $F$-linear span of all values of $\til{f}$. Hence, we need to show that $\til{f}(\til{L}^m,\cdots,\til{L}^m)\subseteq K_a$. Since $\til{L}^m=\sum\limits_i(L^m\otimes e_i+\til{L}^me_i)$, it follows that $\til{f}(\til{L}^m, \cdots, \til{L}^m)=\sum\limits_i\til{f}(L^m\otimes e_i+\til{L}^me_i, \til{L}^m,\cdots, \til{L}^m)$. By our assumption, $\til{f}(L^m\otimes  e_i+\til{L}^me_i, \til{L}^m,\cdots, \til{L}^m)\subseteq K_a'\cap(L^m\otimes e_i+\til{L}^me_i)\subseteq K_a$. This completes the proof of assertion (2).
\end{proof}

The following proposition is a linearized an quadratic version of the construction in \cite{FGG}.

\begin{proposition}\label{Proposition1}
$J_a=(J_a,x\rightarrow x^2, Q)$ is a quadratic Jordan algebra.
\end{proposition}

Since the ground field is infinite partial linearizations of the identities (M1)-(M6) follows from these identities (see \cite{J2},\cite{ZSSS}).

 We will translate the identities (M1)-(M6) into the language of Lie algebras. The identities (M1)-(M6) translate as 
\begin{enumerate}[(M1)]
\item $x\ad^{[2]}(a)\ad(x)\ad(y)=a\ad(a\ad^{[2]}(x))\ad(y)$ mod $K_a$,
\item $y\ad^{[2]}(a)\ad^{[2]}(x)\ad(a)\ad(x)=-a\ad(x)\ad(y)\ad^{[2]}(a)\ad^{[2]}(x)\linebreak=y\ad([a,x])\ad^{[2]}(a)\ad^{[2]}(x)$ mod $K_a$,
\item $a\ad^{[2]}(x)\ad^{[2]}(a)\ad^{[2]}(x)=a\ad^{[2]}(a\ad^{[2]}(x))$ mod $K_a$,
\item $a\ad^{[2]}(x)\ad^{[2]}(a)\ad^{[2]}(y)\ad^{[2]}(a)\ad^{[2]}(x)\\=a\ad^{[2]}(y\ad^{[2]}(a)\ad^{[2]}(x))$ mod $K_a$,
\item $y\ad^{[2]}(a)\ad^{[2]}(a\ad^{[2]}(x))=y\ad^{[2]}(a)\ad^{[2]}(x)\ad^{[2]}(a)\ad^{[2]}(x)$ \linebreak mod $K_a$,
\item $z\ad^{[2]}(a)\ad^{[2]}(y\ad^{[2]}(a)\ad^{[2]}(x))\\=z\ad^{[2]}(a)\ad^{[2]}(x)\ad^{[2]}(a)\ad^{[2]}(y)\ad^{[2]}(a)\ad^{[2]}(x)$ mod $K_a$.
\end{enumerate}

\begin{remark*} In the formulas above, we have operators $\ad^{[2]}(x)$,\linebreak $\ad^{[2]}(a\ad^{[2]}(x))$, $\ad^{[2]}(y\ad^{[2]}(a)\ad^{[2]}(x))$ acting on elements from the space $Fa+\til{L^m}\ad^{[2]}(a)$. In the definition of Jordan operations on $J_a=\til{L^m}/K_a$ above, we noticed that $(Fa+\til{L^m}\ad^{[2]}(a))\ad(\til{L^m})\subseteq K_a$. Hence for an arbitrary element $u\in \{x,a\ad^{[2]}(x),y\ad^{[2]}(a)\ad^{[2]}(x)\}$, the operator $\ad^{[2]}(u)$ is understood as $\sum\ad(u_i)\ad(u_j)$, where $u=\sum u_j$ is the standard decomposition, the sum runs over all 2-element sets $(i,j)$ and the order of factors in $\ad(u_i)\ad(u_j)$ is irrelevant modulo $K_a$.

\end{remark*}

Let $x,y\in\til{L^m}$; $x=\sum\limits_\pi x_\pi$, $y=\sum\limits_\tau y_\tau$ the standard decompositions. At first, we will prove the identities (M1)-(M6) under the additional assumption that $[x_{\pi_i}, x_{\pi_j}]=[y_{\tau_i},y_{\tau_j}]=0$,
\[\ad(x_{\pi_i})\ad(x_{\pi_j})\ad(x_{\pi_k})=\ad(y_{\tau_1})\ad(y_{\tau_2})\ad(y_{\tau_3})=0\]
for all $i,j,k$.

More precisely, let $L_0'$ be the Lie algebra presented by generators $a_1,\cdots, a_n$, $x_1, \cdots, x_s$, $y_1,\cdots, y_t$ and the following relations:
\[[\Id(a_i),\Id(a_i)]=[\Id(x_j), \Id(x_j)]=[\Id(y_k), \Id(y_k)]=(0),\]
where $\Id(a_i)$, $\Id(x_j)$, $\Id(y_k)$ denote the ideals generated by $a_i$, $x_j$, $y_k$ respectively, $1\leq i\leq n$, $1\leq j\leq s$, $1\leq k\leq t$; $[a_i,a_j]=0$, $1\leq i,j\leq n$; the operators $\ad^{[k]}(a)=\sum\ad(a_{i_1})\cdots\ad(a_{i_k})$, where the sum is taken over all $k$-element subsets of $\{1,2,\cdots, n\}$ is equal to 0 for $k\geq 3$.

Denote $a=\sum\limits_{i=1}^na_i$, $x=\sum\limits_{j=1}^sx_j$.

\begin{remark*}
The generators $a_1,\cdots, a_n$ should not be confused with elements $a_1,\cdots, a_r\in\til{L^m}$ used to define $a'=w'(a_1,\cdots, a_{r-1})$,\linebreak $a=w(a_1,\cdots, a_r)$ above.
\end{remark*}

In the algebra $L_0'$, define linear operators
\[\ad^{[2]}(a)=\sum\limits_{1\leq i<j\leq n}\ad(a_i)\ad(a_j),\]
\[\ad^{[2]}(x)=\sum\limits_{1\leq i<j\leq s}\ad(x_i)\ad(x_j),\]
\[\ad^{[2]}(y)=\sum\limits_{1\leq i<j\leq t}\ad(y_i)\ad(y_j),\]
\[\ad^{[2]}(a\ad^{[2]}(x))=\sum\limits_{1\leq i<j\leq n}\ad(a_i\ad^{[2]}(x))\ad(a_j\ad^{[2]}(x)),\]
\[\ad^{[2]}(y\ad^{[2]}(a)\ad^{[2]}(x))=\sum\limits_{1\leq i<j\leq t}\ad(y_i\ad^{[2]}(a)\ad^{[2]}(x))\ad(y_j\ad^{[2]}(a)\ad^{[2]}(x))\]
and consider the elements
\begin{enumerate}[(M1$'$)]
\item $(x\ad^{[2]}(a)\ad(x)\ad(y)-a\ad(a\ad^{[2]}(x))\ad(y)\ad^{[2]}(a)$,
\item $(y\ad^{[2]}(a)\ad^{[2]}(x)\ad(a)\ad(x)\\+a\ad(x)\ad(y)\ad^{[2]}(a)\ad^{[2]}(x))\ad^{[2]}(a)$,
\item $(a\ad^{[2]}(x)\ad^{[2]}(a)\ad^{[2]}(x)-a\ad^{[2]}(a\ad^{[2]}(x)))\ad^{[2]}(a)$,
\item $(a\ad^{[2]}(x)\ad^{[2]}(a)\ad^{[2]}(y)\ad^{[2]}(a)\ad^{[2]}(x)\\-a\ad^{[2]}(y\ad^{[2]}(a)\ad^{[2]}(x)))\ad^{[2]}(a)$,
\item $(y\ad^{[2]}(a)\ad^{[2]}(a\ad^{[2]}(x))\\-y\ad^{[2]}(a)\ad^{[2]}(x)\ad^{[2]}(a)\ad^{[2]}(x))\ad^{[2]}(a)$,
\item $(z\ad^{[2]}(a)\ad^{[2]}(y\ad^{[2]}(a)\ad^{[2]}(x))\\-z\ad^{[2]}(a)\ad^{[2]}(x)\ad^{[2]}(a)\ad^{[2]}(y)\ad^{[2]}(a)\ad^{[2]}(x)\ad^{[2]}(a)$.
\end{enumerate}

Now, consider the Lie algebra $L_0$ that is obtained from $L_0'$ by imposing additional relations:\\
$[x_i,x_j]=0$, $1\leq i,j\leq s$; $[y_i,y_j]=0$, $1\leq i,j\leq t$; $[L_0,x_{i_1},x_{i_2}, x_{i_3}]=[L_0,y_{j_1},y_{j_2},y_{j_3}]=(0)$, for all $1\leq i_1,i_2,i_3\leq s$, $1\leq j_1,j_2,j_3\leq t$.

We will show that the elements (M1$'$)$-$(M6$'$) are equal to zero in the Lie algebra $L_0$.

\begin{lemma}\label{Lemma15}
$[a\ad^{[2]}(x),a]+[x\ad^{[2]}(a),x]\in [L_0,a,a].$
	\begin{proof}
		If $p\neq2$ then $\adx[2][](x)=\frac{1}{2}\ad(x)^2,\ \adx[2][](a)=\frac{1}{2}\ad^2(a),$ which makes the assertion of the lemma obvious.
		
		Let $p=2$.  Denote $a'=a_{i},\ a''=a_{j},\ x'=x_{k},\ x''=x_{e}.$  We will show that
		\[[a',x',x'',a'']+[a'',x',x'',a']=[x',a',a'',x'']+[x'',a',a'',x'].\]
		Indeed, $[a',x',x'',a'']+[a'',x',x'',a']=[[a',x'],[x'',a'']]+[a',x',a'',x'']+[[a'',x'],[x'',a']]+[a'',x',a',x'']=[[a',x'],[x'',a'']]+[[a'',x'],[x'',a']]$, since $[a',x',a'']+[a'',x',a']=[[a',a''],x']=0.$\\
		Similarly, 
	\[[x',a',a'',x'']+[x'',a',a'',x']=[[x',a'],[a'',x'']]+[[x'',a'],[a'',x']],\]
	 which finishes the proof of the lemma.
	\end{proof}
\end{lemma}

Now (M1$'$) immediately follows from Lemma \ref{Lemma15} since \linebreak $L_0\ad(a)^2\ad(y)\subseteq K_a$. The latter inclusion follows from the following argument. The equality (see Lemma \ref{Lemma9} (3)) 

$\begin{array}{r c l}
0&=&\ad(y\ad^{[4]}(a))\\
&=&\ad(y)\ad^{[4]}(a)-\ad(a)\ad(y)\ad^{[3]}(a)+\ad^{[2]}(a)\ad(y)\ad^{[2]}(a)\\
&&-\ad^{[3]}(a)\ad(y)\ad(a)+\ad^{[4]}(a)\ad(y)
\end{array}$

 implies $\ad^{[2]}(a)\ad(y)\ad^{[2]}(a)=0$. Hence, 
\[L_0\ad(a)^2\ad(y)\ad^{[2]}(a)\subseteq L_0\ad^{[2]}(a)\ad(y)\ad^{[2]}(a)=0.\]

Let us prove (M2$'$). From $\ad^{[2]}(x)\ad(a)\ad(x)=\ad(x)\ad(a)\ad^{[2]}(x)$ and $\ad^{[2]}(a)\ad(x)\ad(a)=\ad(a)\ad(x)\ad^{[2]}(a)$ (see Lemma \ref{Lemma9} (3)), it follows that
\begin{align*}
&y\ad^{[2]}(a)\ad^{[2]}(x)\ad(a)\ad(x)=\\
&y\ad^{[2]}(a)\ad(x)\ad(a)\ad^{[2]}(x)=\\
&y\ad(a)\ad(x)\ad^{[2]}(a)\ad^{[2]}(x)=\\
&y\ad([a,x])\ad^{[2]}(a)\ad^{[2]}(x),
\end{align*}
since $\ad(a)\ad^{[2]}(a)=3\ad^{[3]}(a)=0$.

Now we will prove (M3$'$). We have 
\[\ad^{[2]}(a\ad^{[2]}(x))=\sum\ad(a_i\ad^{[2]}(x))\ad(a_j\ad^{[2]}(x)),\] 
where the sum is taken over all 2-element subsets $(i,j)$. By applying Lemma \ref{Lemma14} to $\Omega=\{x_1,\cdots,x_m\}$, $y_1=a_i$, $y_2=a_j$, we get
\[\ad(a_i\ad^{[2]}(x))\ad(a_j\ad^{[2]}(x))=\ad^{[2]}(x)\ad(a_i)\ad(a_j)\ad^{[2]}(x).\]
Hence,
\begin{align*}
&\sum\ad(a_i\ad^{[2]}(x))\ad(a_j\ad^{[2]}(x))=\\
&\ad^{[2]}(x)\Big(\sum\ad(a_i)\ad(a_j)\Big)\ad^{[2]}(x)=\\
&\ad^{[2]}(x)\ad^{[2]}(a)\ad^{[2]}(x)
\end{align*}
as claimed.

Let us prove (M4$'$). We have
\begin{align*}
&\ad^{[2]}(y\ad^{[2]}(a)\ad^{[2]}(x))=\\
&\sum\ad(y_i\ad^{[2]}(a)\ad^{[2]}(x))\ad(y_j\ad^{[2]}(a)\ad^{[2]}(x)).
\end{align*}
By Lemma \ref{Lemma14}, with $\Omega=\{x_1,\cdots, x_m\}$, we get
\begin{align*}
&\ad(y_i\ad^{[2]}(a)\ad^{[2]}(x))\ad(y_j\ad^{[2]}(a)\ad^{[2]}(x))=\\
&\ad^{[2]}(x)\ad(y_i\ad^{[2]}(a))\ad(y_j\ad^{[2]}(a))\ad^{[2]}(x).
\end{align*}
Again, by Lemma \ref{Lemma14} with $\Omega=\{a_1, \cdots, a_n\}$
\[\ad(y_i\ad^{[2]}(a))\ad(y_j\ad^{[2]}(a))=\ad^{[2]}(a)\ad(y_i)\ad(y_j)\ad^{[2]}(a).\]
Finally, we get 
\[\ad^{[2]}(y\ad^{[2]}(a)\ad^{[2]}(x))=\ad^{[2]}(x)\ad^{[2]}(a)\ad^{[2]}(y)\ad^{[2]}(a)\ad^{[2]}(x),\] 
as claimed.

We will now prove (M5$'$). We have already shown above that by Lemma \ref{Lemma14}, we have 
\[\ad^{[2]}(a\ad^{[2]}(x))=\ad^{[2]}(x)\ad^{[2]}(a)\ad^{[2]}(x),\]
which implies the claim.

To prove (M6$'$), we need only to recall the equality
\[\ad^{[2]}(y\ad^{[2]}(a)\ad^{[2]}(x))=\ad^{[2]}(x)\ad^{[2]}(a)\ad^{[2]}(y)\ad^{[2]}(a)\ad^{[2]}(x)\]
that was proved above.

Since the elements (M1$'$)$-$(M6$'$) are equal to zero in $L_0$, it follows that in the algebra $L_0'$, the elements of (M1$'$)$-$(M6$'$) are linear combinations of
\begin{enumerate}[(1)]
\item expressions in $x_i$'s, $y_j$'s, $z$, $a_1,\cdots, a_n$ involving at least one commutator $[x_i,x_j]$, $1\leq i,j\leq s$ or $[y_i,y_j]$, $1\leq i,j\leq t$,
\item expressions involving $\ad(x_{i_1})\ad(x_{i_2})\ad(x_{i_3})$ or $\ad(y_{j_1})\ad(y_{j_2})\ad(y_{j_3})$, $1\leq i_1,i_2,i_3\leq s$, $1\leq j_1,j_2,j_3\leq t$.
\end{enumerate}
Moreover, since the relations of the algebra $L_0'$ are homogeneous in $x_i$'s, $y_j$'s, and in the total number of generators $a_1,\cdots, a_n$, it follows that the presentations of (M1$'$)$-$(M6$'$) as linear combinations of (1), (2) preserve the degrees in $x_i$'s, $y_j$'s, and the total degree in $a_1,\cdots, a_n$.

Now we consider arbitrary elements $x,y\in\til{L^m}$ and drop the assumptions on components of standard decompositions of $x,y$. Let $x=x_{\pi_1}+\cdots+x_{\pi_s}$, $y=y_{\tau_1}+\cdots+y_{\tau_s}$, $a_r=a_{r1}+\cdots+a_{rn}$ be the standard decompositions of $x$, $y$, $a_r$ respectively. Then $a=\sum\limits_{i=1}^na_i$, where $a_i=a_{ri}\ad^{[2]}(a')$.

The mapping $a_i\rightarrow a_{ri}\ad^{[2]}(a')$, $x_j\rightarrow x_{\pi_j}$, $y_k\rightarrow y_{\tau_k}$, $1\leq i\leq n$, $1\leq j\leq s$, $1\leq k\leq t$, extends to a homomorphism $L_0'\rightarrow\til{L^m}$. Moreover, the operators $\ad^{[2]}(a)$, $\ad^{[2]}(x)$, $\ad^{[2]}(y)$, $\ad^{[2]}(a\ad^{[2]}(x))$, $\ad^{[2]}(y\ad^{[2]}(a)\ad^{[2]}(x))$ project to the similar operators on $\til{L^m}$, by Lemma \ref{Lemma40}.

Hence, the elements (M1$'$)$-$(M6$'$) of $\til{L^m}$ are linear combinations of
\begin{enumerate}[(1)]
\item expressions in $x_{\pi_i}$, $y_{\tau_j}$, $a_1,\cdots, a_n$ involving at least one commutator $[x_{\pi_i},x_{\pi_j}]$ or $[y_{\tau_i},y_{\tau_j}]$,
\item expressions involving $\ad(x_{\pi_i})\ad(x_{\pi_j})\ad(x_{\pi_k})$ or $\ad(y_{\tau_i})\ad(y_{\tau_j})\ad(y_{\tau_k})$.
\end{enumerate}
These presentations, as linear combinations of (1) and (2), preserve the degrees in $x_{\pi_i}$'s, $y_{\tau_j}$'s and the total degree in $a_1,\cdots, a_n$.

Replacing $\ad(a_i)$ in these expressions by 
\[\ad(a_{ri}\ad^{[2]}(a'))=\ad(a_{ri})\ad^{[2]}(a')-\ad(a')\ad(a_{ri})\ad(a')+\ad^{[2]}(a')\ad(a_{ri})\]
we get expressions whose degree in $a'$ exceeds the total degree in the other variables $x_{\pi_i}$, $y_{\tau_j}$, $z$, $a_{ri}$ by 1. In case (1), we merge two elements $x_{\pi_i}$, $x_{\pi_j}$ or $y_{\tau_i}$, $y_{\tau_j}$ together. Hence, the degree in $a'$ exceeds the total degree in the other elements by 2. The only property of the element $a$ that was used in Lemma \ref{Lemma13} was $\ad^{[k]}(a)=0$ for $k\geq 3$. We have $\ad^{[k]}(a')=0$, $k\geq 3$. Hence, we can apply Lemma \ref{Lemma13} to the element $a'$. By Lemma \ref{Lemma13} (1), these expressions are equal to zero. In case (2), we only need to refer to Lemma \ref{Lemma13} (3). We proved that the expressions (M1$'$)$-$(M6$'$) are equal to 0, which means that the expressions (M1)$-$(M6) are equal to 0 modulo $K_a'$. By Lemma \ref{Lemma41} (2), they are equal to 0 modulo $K_a$, which finishes the proof of Proposition \ref{Proposition1}.

Let us consider basic examples of quadratic Jordan algebras.
	
\underline{Example 1.}  Let $A$ be an associative algebra.  Let $yQ(x)=xyx;\ x,y\in A.$  Then the vector space $A$ with the operators $x\rightarrow x^2$ and $x\rightarrow Q(x)$ is a quadratic Jordan algebra, which is denoted as $A^{(+)}$.
	
\underline{Example 2.}  Let $A$ be an associative algebra with an involution $\ast:A\rightarrow A$.  Then $H(A,\ast)=\{a\in A|a^\ast=a\}$ is a subalgebra of the quadratic Jordan algebra $A^{(+)}.$
	
\underline{Example 3.} Let $V$ be a vector space and let $q:V\rightarrow F$ be a quadratic form with the associated bilinear form $q(v,w)=q(v+w)-q(v)-q(w)$. Fix an element of $V$ that we will denote as $\bb1$ (a base point) such that $q(\bb1)=1$.  For arbitrary elements $v,w\in V$ define
	\[v^2=q(v,\bb1)v-q(v)\bb1,\ wQ(v)=q(v,\bar{w})v-q(v)\bar{w},\]
	where $\bar{w}=q(w,\bb1)\bb1-w$.  These equations make $V$ a quadratic Jordan algebra.  We will denote it as $J(q,\bb1).$
	
	\underline{Example 4.} Albert algebras of a nondegenerate admissible cubic form on a 27-dimensional space (see \cite{J2,J4}).
	
	Powers of elements in a quadratic Jordan algebra $J$ are defined inductively: we define $x^1=x$; for an even $n=2k$ we define $x^n=(x^k)^2;$ and for an odd $n=2k+1$ we define $x^n=xQ(x^k).$  For arbitrary integers $i\geq0,j\geq0,k\geq0$ we have $x^iQ(x^j)=x^{i+2j},\ x^i\circ x^j=2x^{i+j},\ \{x^i,x^j,x^k\}=2x^{i+j+k}.$ 
	
	A quadratic Jordan algebra $J$ is said to be nil of bounded degree $n$ if $x^n=0$ for an arbitrary element $x\in J$ and if $n$ is a minimal integer with this property.
	
	Just as in $\mathsection2$ we call an element of the free quadratic Jordan algebra $FJ\langle X\rangle$ an $S$-identity if it lies in the kernel of the homomorphism $FJ\langle X\rangle\rightarrow F\langle X\rangle^{(+)},\ x\rightarrow x,$ where $F\langle X\rangle$ is the free associative algebra.
	
	We say that a quadratic Jordan algebra $J$ is PI if there exists an element $f(x_1,\cdots,x_r)\in FJ\langle X\rangle$ that is not an $S$-identity and that is identically zero on $J$.
	
In this paper, we call an element $a$ of a quadratic Jordan algebra an \underline{absolute zero divisor} if $Q(a)=0$. This terminology is not standard. (In the standard terminology, we should have also assumed $a\neq 0$ and $a^2=0$.)  A quadratic Jordan algebra that does not contain nonzero absolute zero divisors is called nondegenerate.  The smallest ideal $M(J)$ of a Jordan algebra $J$ such that the factor algebra $J/M(J)$ is nondegenerate is called the McCrimmon radical of the algebra $J$.  The McCrimmon radical of an arbitrary quadratic Jordan algebra lies in the nil radical Nil$(J)$ (\cite{Z1,M3}).
	
	A nondegenerate quadratic Jordan algebra is said to be nondegenerate prime if two arbitrary nonzero ideals of $J$ have nonzero intersection.
	
	In \cite{Z3,Th} it is shown that an arbitrary nondegenerate Jordan algebra is a subdirect product of nondegenerate prime Jordan algebras.
	
	Let Sym$_n(x_1,\cdots,x_n)$ be the full linearization of $x_1^n$ in the free Jordan algebra $FJ\langle X\rangle$.
	
	\begin{lemma}
		\label{Lemma16} There exists a function $d:\bbn\rightarrow \bbn$ such that an arbitrary nondegenerate prime quadratic Jordan algebra over a field of characteristic $p>0$ satisfying a PI of degree $n$ satisfies the identity Sym$_{d(n)}(x_1,\cdots,x_{d(n)})=0.$
		\begin{proof}
			Let us notice first that if $J$ is a quadratic Jordan algebra of dimension $d$ then $J$ satisfies the identity Sym$_{d(p-1)+1}(x_1,\cdots,x_{d(p-1)+1})=0$.  Indeed, if $e_1,\cdots,e_d$ is a basis of $J$ then among any $d(p-1)+1$ elements from $\{e_1,\cdots,e_d\}$ at least $p$ elements are equal.  This implies the claim.
			
			In \cite{MZ} it was shown that if $J$ is a nondegenerate prime quadratic Jordan algebra, then one of the following possibilities holds:
			\begin{enumerate}[(1)]
				\item there exists a prime associative algebra $A$ such that \[A^{(+)}\subseteq J\subseteq \mathds{Q}(A)^{(+)},\]
				 where $\mathds{Q}(A)$ is the Martindale ring of the quotients of $A$ (see \cite{Ma});
				\item there exists a prime associative algebra $A$ with an involution \linebreak$\ast: A\rightarrow A,$ such that 
				\[H(A_0,\ast)\subseteq J\subseteq H(\mathds{Q}(A),\ast)\]
				where $A_0$ is the subalgebra of $A$ generated by elements $a+a^\ast,\ aa^\ast,$ $ a\in A,$ and $\mathds{Q}(A)$ is the Martindale ring of quotients of the algebra $A$ (see \cite{Ma});
			\item $J$ is a form of an exceptional 27-dimensional Albert algebra over a field $F$;
				\item $J$ is embeddable in a quadratic Jordan algebra $J(q,v_0)$ of a nondegenerate quadratic form $q$ with a basepoint $v_0$ in a vector space over some extension of the base field $F$.
			\end{enumerate}		
			If $A^{(+)}\subseteq J\subseteq\mathds{Q}(A)^{(+)}$, then $A$ is a prime associative algebra satisfying an identity of degree $n$.  Hence the center $Z(A)$ of $A$ is nonzero and the algebra $\mathds{Q}(A)=(Z(A)\setminus \{0\})^{-1}A$ is of dimension $\leq [\frac{n}{2}]^2$ over the field $K=(Z(A)\setminus\{0\})^{-1}Z(A)$ (see \cite{R}).  Hence the algebra $\mathds{Q}(A)$ satisfies the identity Sym$_{[\frac{n}{2}]^2(p-1)+1}=0$.
			
			Suppose that $H(A_0,\ast)\subseteq J\subseteq H(\mathds{Q}(A),\ast).$ S. Amitsur \cite{A} proved that there exists a function $h(n)$ with the following property:
			
			if an involutive associative algebra satisfies an identity of degree $n$ with an involution then it satisfies an identity of degree $\leq h(n)$. As we have shown above, the algebra $\mathds{Q}(A)$ in this case has dimension $\leq[\frac{h(n)}{2}]^2$ over its center and satisfies the identity Sym$_{[\frac{h(n)}{2}]^2(p-1)+1}=0.$
			
			The same argument applies to case (3): the algebra $J$ satisfies the identity Sym$_{27(p-1)+1}=0.$
			
			Consider now the quadratic Jordan algebra $J$ of a quadratic form $q$ on a vector space $V$ where $v_0\in V$ is a basepoint.  The quadratic form $q$ can be extended to the scalar product $V\otimes_F\widehat{E},\ \widehat{E}=E+F\cdot1$.  For an arbitrary element $u\in V\otimes_F{E}$ we have $u^2=q(u,v_0)u-q(u)v_0.$  The elements $a=q(u,v_0),\ b=q(u)$ lie in $E$.  Hence $a^p=b^p=0.$  For an arbitrary $k\geq 1$ we have
			\[u^{2k}=\sum\limits_{i+j=k}a^ib^ju_{ij},\ u_{ij}\in V\otimes_F\widehat{E}.\]
			Hence $u^{2(2p-1)}=0.$  This implies that the algebra $J$ satisfies the identity Sym$_{4p-2}$=0 and finishes the proof of the lemma.
		\end{proof}
	\end{lemma}

\begin{lemma}
\label{Lemma17}
Let $J$ be a quadratic Jordan $F$-algebra that satisfies the identity $x^n=0$, $n\geq 2$. Then,
\begin{enumerate}[(1)]
\item \label{thing1} for an arbitrary element $a\in J$, the elements $a^{n+1}, a^{n+2}, \cdots, a^{2n-1}$ are absolute zero divisiors of $J$;
\item \label{thingg2} if $J$ satisfies the identities $x^n=x^{n+1}=\cdots=x^{2n-1}=0$, then for an arbitrary $a\in J$, the element $a^{n-1}$ is an absolute zero divisor of $J$.
\end{enumerate}
\end{lemma}
\begin{proof}
For $i=n+1,n+2,\cdots, 2n-1$, we have $Q(x^i)=Q(x^{i-n})Q(x^n)=0$, which proves (\ref{thing1}).

Suppose now that the algebra $J$ satisfies the identities $x^{n+1}=x^{n+2}=\cdots=x^{2n-1}=0$. Since the ground field $F$ is infinite, the algebra $J$ satisfies also the following partial linearization of $x^{2n-1}=0$ (see \cite{J2}, \cite{ZSSS}):
\[yQ(x^{n-1})+x^{2(n-1)}\cdot y+\sum\limits_{\substack{i+j=2(n-1)\\ 1\leq i<j\leq 2n-3}}\{x^i,y,x^j\}=0.\]
Hence, $yQ(x^{n-1})=0$, which proves assertion \label{thing2} of the lemma.
\end{proof}


Let $J$ be a quadratic Jordan algebra, $a\in J.$  Define a new structure of a quadratic Jordan algebra on $J$ via:
\[x^{\ast 2}=aQ(x),\ yQ^\ast(x)=yQ(a)Q(x).\]
The new quadratic Jordan algebra is denoted as $J^{(a)}$ and is called a homotope of $J$ (see \cite{J2,J4,M3}).

For the quadratic Jordan algebra $\til{J}=J\otimes_F E$ and an element $a\in\til{J}$, consider the subspaces $K_a'=\{x\in\til{J}|xQ(a)=0\}$ and $K_a=\sum\limits_i(J\otimes e_i+\til{J}e_i)\cap K_a'$. It is easy to see that the subspace $K_a$ is an ideal of the algebra $\til{J}^{(a)}$.

\begin{remark*}
If $p\neq 2$, then $K_a'$ is also an ideal of $\til{J}^{(a)}$.
\end{remark*}


\begin{lemma}
\label{Lemma18}
If $b\in J$ and $b+K_a$ is an absolute zero divisor of the algebra $J^{(a)}/K_a$, then $bQ(a)$ is an absolute zero divisor of the algebra $J$.
\end{lemma}
\begin{proof}
We have $JQ(bQ(a))=JQ^\ast(b)Q(a)\subseteq K_aQ(a)=(0)$, which proves the lemma.
\end{proof}

\begin{lemma}
\label{Lemma42}
Let $a$ be an element of a quadratic Jordan algebra $J$. Let $f:\til{J}\times\cdots\times\til{J}\rightarrow\til{J}$ be a homogeneous polynomial map, let $\til{f}$ be its full linearization. Suppose that $\til{f}(J\otimes e_i+\til{J}e_i, \til{J}, \cdots, \til{J})\subseteq J\otimes e_i+\til{J}e_i$ for all $i$. If an arbitrary value of $f$ lies in $K_a'$, then an arbitrary value of $f$ lies in $K_a$.
\end{lemma}

The proof is similar to the proof of Lemma \ref{Lemma41} (2).


Let $f$ be an element of the free quadratic Jordan algebra $FJ\langle X\rangle,$ which is not an $S$-identity.  Let $M=M(f)$ be the variety of quadratic Jordan algebras satisfying the identity $f=0$ (see \cite{J2,J3,ZSSS}).

\underline{Definition.} We say that a finite sequence of homogeneous elements $h_1,h_2,\cdots,h_r\in FJ\langle X\rangle$ is an absolute zero divisor sequence for $M$ if for an arbitrary quadratic Jordan algebra $J\in M$
\begin{enumerate}[(i)]
	\item every value of $h_r$ on $\til J=J\otimes_F E$ is an absolute zero divisor of the algebra $\til J$;
	\item if $h_k=h_{k+1}=\cdots=h_r=0$ identically hold on $\til J,\ 2\leq k\leq r,$ then an arbitrary value of $h_{k-1}$ on $\til J$ is an absolute zero divisor of $\til J.$
\end{enumerate}

Recall that in this section we always assume that char$F=p>0$.

\begin{proposition}
	\label{Proposition2}
	For an arbitrary element $f\in FJ\langle X\rangle$ that is not an $S$-identity the variety $M(f)$ has a finite absolute zero divisor sequence $h_1,h_2,\cdots,h_r$ with $h_1=x_1\in X.$
	\begin{proof}
Let $F_M\langle X\rangle $ be the free algebra in the variety $M=M(f)$ on the set of free generators $X$. Since the factor algebra of $F_M\langle X\rangle$ modulo the McCrimmon radical can be approximated by prime nondegenerate algebras (\cite{Z3}, \cite{Th}), Lemma \ref{Lemma16} implies that there exists $d\geq 1$ such that $y=\text{Sym}_d(x_1,\cdots, x_d)$ lies in the McCrimmon radical of $F_M\langle X\rangle$. Consider the homotope algebra $F_M\langle X\rangle^{(x_{d+1})}$ Since an absolute zero divisor of a Jordan algebra is an absolute zero divisor of every homotope, it follows that the McCrimmon radical of $F_M\langle X\rangle$ lies in the McCrimmon radical of $F_M\langle X\rangle^{(x_{d+1})}$. In particular, the element $y$ lies in the McCrimmon radical of $F_M\langle X\rangle^{(x_{d+1})}$ and therefore is nilpotent.

Let $a^{(k, x_{d+1})}$ denote the $k^{th}$ power of an element $a$ in the homotope algebra $F_M\langle X\rangle^{(x_{d+1})}$. There exists $m\geq 2$ such that $y^{(m-1,x_{d+1})}=0$. Then $x_{d+1}^{(m,y)}=y^{(m-1,x_{d+1})}Q(x_{d+1})=0$. This implies that $x^{(m,y)}=0$ holds identically on $F_M\langle X\rangle$.
		
Then by Lemmas \ref{Lemma17} and \ref{Lemma18}, the sequence
\[y, x_{d+1}Q(y), \cdots, x_{d+1}^{(m-1,y)}Q(y), x_{d+1}^{(2m-1,y)}Q(y), \cdots, x_{d+1}^{(m+1,y)}Q(y)\]
is an absolute zero divisor sequence in $M$.

Indeed, since the Jordan algebra $F_M\langle X\rangle ^{(y)}/K_y$ satisfies the identity $x^m=0$, Lemma \ref{Lemma17} (\ref{thing1}) implies that the elements $x_{d+1}^{(2m-1,y)}+K_y,$ $\cdots, x_{d+1}^{(m+1,y)}+K_y$ are absolute zero divisors in $F_M\langle X\rangle ^{(y)}/K_y$. By Lemma \ref{Lemma18}, the elements $x_{d+1}^{(2m-1,y)}Q(y),\cdots, x_{d+1}^{(m+1,y)}Q(y)$ are absolute zero divisors of the algebra $F_M\langle X\rangle$.

If $J\in M$ and $x_{d+1}^{(2m-1,y)}Q(y)=\cdots=x_{d+1}^{(m+1,y)}Q(y)=0$ hold identically on $\til{J}$, then for arbitrary elements $a_1,\cdots, a_d\in\til{J}$, \linebreak $b=\text{Sym}_d(a_1,\cdots, a_d)$, $c\in\til{J}$ the $i$-th power $c^{(i,b)}$, $m\leq i\leq 2m-1$, lies in $K_b'$. By Lemma \ref{Lemma42}, we have $c^{(i,b)}\in K_b$. In other words, the Jordan algebra $\til{J}^{(b)}/K_b$ satisfies the identities $x^m=x^{m+1}=\cdots=x^{2m-1}=0$. By Lemma \ref{Lemma17} (2), for an arbitrary element $c\in\til{J}$, the $(m-1)$-th power $c^{(m-1,b)}$ is an absolute zero divisor in $\til{J}^{(b)}/K_b$. By Lemma \ref{Lemma18}, the element $c^{(m-1,b)}Q(b)$ is an absolute zero divisor of $\til{J}$. 

If $J\in M$ and $x_{d+1}^{(m-1,y)}Q(y)=x_{d+1}^{(2m-1,y)}Q(y)=\cdots=x_{d+1}^{(m+1,y)}Q(y)=0$ holds identically on $\til{J}$, then using Lemma \ref{Lemma42} as above, we conclude that the algebra $\til{J}^{(b)}/K_b$ satisfies the identities $x^{m-1}=\cdots=x^{2m-1}=0$.

Again, by Lemma \ref{Lemma17} (2) and Lemma \ref{Lemma18}, every value of $x_{d+1}^{(m-2,y)}Q(y)$ (and so on) is an absolute zero divisor of $\til{J}$.


		
		If an algebra $J$ lies in $M$ and $y=\text{Sym}_d(x_1,\cdots,x_d)=0$ holds identically on $J$ then the algebra $\til J$ is nil of bounded index $\leq d.$  Again by \Cref{Lemma17} we conclude that 
\begin{align*}
x_1,x_1^2,\cdots,x_1^{d-1},x_1^{2d-1},\cdots, x_1^{d+1}, y, x_{d+1}Q(y),\cdots,\\
 x_{d+1}^{(m-1,y)}Q(y),x_{d+1}^{(2m-1,y)}Q(y),\cdots, x_{d+1}^{(m+1,y)}Q(y)
\end{align*}
		is an absolute zero divisor sequence, which finishes the proof of the proposition.		
	\end{proof}
\end{proposition}

\begin{conjecture}
	If $J$ is a quadratic Jordan PI-algebra over a field of characteristic $p>0$ then the algebra $\til J$ is nil of bounded index.
\end{conjecture}

Now let's come back to the Lie algebra $L$ and the Jordan algebra $J_a=\til L^m/K_a.$

\begin{lemma}
	\label{Lemma19}
	Let $b+K_a$ be a nonzero absolute zero divisor of the Jordan algebra $\til{L}^m/K_a$. Then the element $b\ad^{[2]}(a)$ is a nonzero sandwich of the Lie algebra $\til{L}^m.$
	\begin{proof}
		Let $b = \sum\limits_{\pi} b_\pi$ be the standard decomposition. For an arbitrary element $c \in \til{L}^m$ we have 
		\[(c+K_a)Q(b+K_a) = \sum c \ad^{[2]}(a)\ad(b_{\pi_1})\ad(b_{\pi_2})+K_a.\]
		Hence by Lemma \ref{Lemma14},
		\[\sum c \ad^{[2]}(a)\ad(b_{\pi_1})\ad(b_{\pi_2})\ad^{[2]}(a) \]
		\[=\sum c\ad(b_{\pi_1}\ad^{[2]}(a))\ad(b_{\pi_2}\ad^{[2]}(a)) = 0.\]
		Let $\Omega =\{b_{\pi_i}\ad^{[2]}(a)\}.$ We showed that $\til{L}^mU_2(\Omega)=(0).$ By Lemma \ref{Lemma12} the element $b\ad^{[2]}(a)$ is a sandwich of the Lie algebra $\til{L}^m.$
	\end{proof}
\end{lemma}

Let $j(y_1, \cdots, y_d)$ be an arbitrary Jordan polynomial, i.e., an element of the free Jordan algebra. The polynomial $j$ defines a function $\til{L}^m/K_a \times \cdots \times \til{L}^m/K_a \rightarrow \til{L}^m/K_a$ and, therefore, a function $\til{L}^m \times \cdots \times \til{L}^m \rightarrow \til{L}^m/K_a.$

\begin{lemma}
\label{Lemma35}
Let $j(y_1, \cdots, y_q)$ be a multilinear Jordan polynomial. There exists a homogeneous divided polynomial $j'(y_1, \cdots, y_q, x_1, \cdots, x_r)$ defined on $L^m,$ such that the value $j(b_1, \cdots, b_q)$ in the Jordan algebra $\til{L}^m/K_a$ is equal to $j'(b_1, \cdots, b_q, a_1, \cdots, a_r) + K_a.$ In particular, 
\[j(b_1, \cdots, b_q)\text{ ad}^{[2]}w(a_1, \cdots, a_r)\] 
\[= j'(b_1, \cdots, b_q, a_1, \cdots, a_r)\text{ ad }^{[2]}w(a_1, \cdots, a_r).\]
\begin{proof}
We will proceed by induction on the construction of the Jordan polynomial $j.$ Let $j=\alpha j_1 + \beta j_2,$ where $\alpha, \beta \in F$ and $j_1, j_2$ are multilinear Jordan polynomials, such that the divided polynomials $j'_1, j_2'$ exist. Then we let $j'=\alpha j_1'+\beta j_2'.$ Suppose that $j=j_1\circ j_2$, where $j_1, j_2$ are multilinear Jordan polynomials on disjoint variables. 

We have $j(b_1,\cdots, b_q)=[a,j_1(b_1,\cdots,b_q), j_2(b_1,\cdots,b_q)]+K_a$ and we let
\begin{align*}
&j'(y_1,\cdots,y_q,x_1,\cdots,x_r)=\\
&[w(x_1,\cdots,x_r),j_1'(y_1,\cdots,y_q,x_1,\cdots,x_r),j_2'(y_1,\cdots,y_q,x_1,\cdots,x_r)].
\end{align*}
Finally, let $j=\{j_1,j_2,j_3\},$ where $j_1, j_2, j_3$ are multilinear Jordan polynomials on disjoint variables. Arguing as above, we let
\[j'=[j_2'\ad_{x_1}^{[2]}w(x_1,\cdots,x_r),j_1',j_3'].\]
This completes the proof of the lemma.
\end{proof}
\end{lemma}

\begin{proposition}
	\label{Proposition3}
	There exist integers $k \geq 1,$ $t\geq 1$ and a homogeneous regular divided polynomial $v$ defined on $L^k$ such that every value of $v$ on $\til{L}^k$ is a sum of $t$ sandwiches of the algebra $\til{L}^k.$
	\begin{proof}
		Recall that there exists a homogeneous regular divided polynomial $w=w(x_1, \cdots, x_r)$ defined on $L^m,$ $m \geq 1,$ linear in $x_r$ and  such that
		\begin{enumerate}[(i)]
			\item $[w(a_1,\cdots,a_{r-1},a), w(a_1,\cdots, a_{r-1},b)]=0$ for arbitrary elements \linebreak $a, b, a_1, \cdots, a_{r-1} \in \til{L}^m;$
			\item $\til{L}^m\text{ ad}_{x_r}^{[t]}(w)=0$ holds identically on $\til{L}^m$ for $t \geq 3.$
		\end{enumerate}
		For an arbitrary $i\geq m,$ arbitrary elements $a_1,\cdots,a_r\in \til L^i$ consider $a=w(a_1,\cdots,a_r)$ and denote $\adx[2][](a)=\ad_{x_r}^{[2]}(w(a_1,\cdots,a_r)).$
		

Aruging as in the proof of Lemma \ref{Lemma2}, we conclude that there exists an element $f\in FJ\langle X\rangle$ such that $f$ is not an $S$-identity and all quadratic Jordan algebras $\til L^i/K_a;\ i\geq m;\ a_1,\cdots,a_r\in \til L^i$ satisfy the identity $f=0.$
		
		By \Cref{Proposition2} there exists an absolute zero divisor sequence $h_1=x_1,h_2,\cdots,h_s$ of the variety $M(f).$

If $J$ is an algebra from the variety $M(f)$ such that $J=\sum\limits_i I_i$, $I_i\unlhd J$, $I_i^2=(0)$, then $J$ and $J\otimes_F E$ satisfy the same identities. Hence, every value of $h_s$ on $J$ is an absolute zero divisor of $J$ and if $h_k=\cdots=h_s=0$ identically hold on $J$, then every value of $h_{k-1}$ on $J$ is an absolute zero divisor. 

Jordan algebras $\til{L^m}/K_a$ that have been discussed above have this property. Indeed, $\til{L^m}/K_a=\sum\limits_i I_i$, where $I_i=L^m\otimes e_i+\til{L^m}e_i+K_a/K_a$.
		
		For an integer $i\geq m$ and elements $a_1,\cdots,a_r\in\til L^i$ let $s(i,a_1,\cdots,a_r)$ be a maximal integer $j,\ 1\leq j\leq s,$ such that $h_j$ is not identically zero on $\til L^i/K_a.$  If $h_1=x_1$ is identically zero on $\til L^i/K_a,$ that is, $\til L^i=K_a,$ then we let $s(i,a_1,\cdots,a_r)=0.$
		
		Let $s(i)=\max\{s(i,a_1,\cdots,a_r)|a_1,\cdots,a_r\in \til L^i\}.$  Clearly, $s(m)\geq s(m+1)\geq \cdots.$  Let this decreasing sequence stabilize at $t=s(k)=s(k+1)=\cdots.$
		
		If $t=0$ then $\til{L^k}=K_a$, which means that $\til{L^k}\ad_{x_1}^{[2]}w(a_1,\cdots,a_r)=(0)$ for arbitrary elements $a_1,\cdots,a_r\in \til{L^k}$. By Lemma \ref{Lemma12} every value of $w$ on $\til{L^k}$ is a sandwich of the algebra $\til{L^k}$. Therefore assume that $t \geq 1.$

Let us summarize the above. For arbitrary elements $a_1, \cdots, a_r \in \til{L^k}$ let $a=w(a_1, \cdots, a_r)$, $\ad^{[2]}(a)=\ad_{x_r}^{[2]}w(a_1, \cdots, a_r)$, $K_a=\til{L^k}\cap ker\ad^{[2]}(a)$. Every value of the Jordan polynomial $h_t$ on the Jordan algebra $\til{L^k}/K_a$ is an absolute zero divisor. For every $k'\geq k$ there exist elements $a_1, \cdots, a_r \in \til{L^{k'}}$ such that $h_t$ is not identically zero on $\til{L^{k'}}/K_a$. In particular, the Jordan polynomial $h_t$ is regular.

Suppose that $h_t=h_t(y_1,\cdots, y_q)$. Let $\mu$ be the total degree of the homogeneous Jordan polynomial $h_t$. The full linearization $\til{h_t}$ of the polynomial $h_t$ depends on $\mu$ variables. An arbitrary value of the polynomial $\til{h_t}$ is a linear combination of $2^\mu=\ell$ values of the polynomial $h_t$. Let $\til{h_t}'(y_1, \cdots, y_q, x_1, \cdots, x_r)$ be the homogeneous divided polynomial of Lemma \ref{Lemma35} defined on $\til{L^k}$. Let $v(y_1,\cdots, y_q, x_1, \cdots, x_r)=\til{h_t'}(y_1, \cdots, y_q, x_1, \cdots, x_r)\ad_{x_r}^{[2]}w(x_1, \cdots, x_r)$.

For arbitrary elements $b_1, \cdots, b_q, a_1, \cdots, a_r \in \til{L^k}$, we have
\[\til{h_t}(b_1, \cdots, b_q)\ad_{x_r}^{[2]}w(a_1, \cdots, a_r)=v(b_1, \cdots, b_q, a_1, \cdots, a_r).\]

We claim that the divided polynomial $v$ is regular. Indeed, it was shown above that for arbitrary $k'\geq k$, there exist elements $a_1,\cdots, a_r\in\til{L^{k'}}$ such that the Jordan polynomial $h_t$ is not identically zero on $\til{L^{k'}}/K_a$. Lemma \ref{Lemma33} (2) was proved for arbitrary polynomial maps that include Jordan polynomials. Hence by Lemma \ref{Lemma33} (2), the linear spans of the sets of values of the Jordan polynomials $h_t$ and $\til{h_t}$ on the Jordan algebra $\til{L^{k'}}/K_a$ are equal. Hence $\til{h_t}$ is not identically zero on $\til{L^{k'}}/K_a$. By Lemma \ref{Lemma35}, the homogeneous divided polynomial $v=\til{h_t}\ad_{x_r}^{[2]}(w)$ is not identically zero on $\til{L^{k'}}$. This implies regularity of $v$.

By Lemma \ref{Lemma19} every value of $v$ on $\til{L^k}$ is a sum of $\ell=2^\mu$ sandwiches of the Lie algebra $\til{L^k}$. This completes the proof of the proposition.
	\end{proof}
\end{proposition}

\section{Sandwiches in $\til L$}
\label{Section5}

Let $x\in \til L,\ x=\sum\limits_\pi x_\pi$ the standard decomposition.  Suppose that
\begin{equation}
[x_\pi,x_\tau]=0 \label{Star1}
\end{equation}
for arbitrary $\pi,\tau.$  As above denote $\adx[k][](x)=\sum\ad(x_{\pi_1})\cdots\ad(x_{\pi_k}),$ where the sum runs over all $k$-element subsets $(\pi_1,\cdots,\pi_k).$  As we have already noticed in \Cref{Lemma9}(2), $A(x)=\text{Id}+\sum\limits_{k=1}^\infty \adx[k][](x)$ is an automorphism of the algebra $\til L.$

Let elements $x_1,\cdots,x_d\in\til L$ satisfy condition \ref{Star1}, \linebreak$A=A(x_1)\cdots A(x_d)\in \Aut \til L.$  The following lemma is straightforward.

\begin{lemma}
	\label{Lemma20}
	For arbitrary elements $a_1,\cdots,a_d\in\til L$ we have
	\begin{align*}
	&	a_1A\otimes\cdots\otimes a_dA-\sum\limits_{i=1}^d a_1A\otimes\cdots a_i\otimes\cdots\otimes a_dA+\\
	&	\sum\limits_{1\leq i\neq j\leq d}a_1A\otimes\cdots\otimes a_i\cdots\otimes a_j\cdots\otimes a_dA-\cdots\pm a_1\otimes\cdots\otimes a_d=\\
	&	\sum\limits_{\sigma\in S_d}[a_1,x_{\sigma(1)}]\otimes\cdots\otimes[a_d,x_{\sigma(d)}]\\
	&	+\text{terms involving at least two elements from one of the sets } \{x_{i_\pi}\}_\pi.
	\end{align*}
\end{lemma}

A. N. Grishkov \cite{Gri} showed that in a Lie algebra over a field of zero characteristic a sandwich generates the locally nilpotent ideal.  We will prove an analog of this result for the algebra $\til L$.  The proof essentially depends on the following result from \cite{KZ}:

There exists a function $KZ:\bbn\rightarrow \bbn$ such that in an arbitrary Lie algebra if elements $a_1,\cdots,a_n$ are sandwiches then the subalgebra $\langle a_1\cdots,a_n\rangle$ is nilpotent of degree $\leq KZ(n).$

Let $f(m,n)=KZ((n+1)^m)$.

\begin{lemma}
	\label{Lemma21}
	Let $I$ be an ideal of a Lie algebra $L,\ a\in\til I$ a sandwich in $\til I.$  Let $S\subset\til L$ be a finite set of $\leq n$ elements.  Then the subalgebra of $\til L$ generated by commutators $[a,b_1\cdots,b_t],$ where $b_i\in S,\ t\leq m,$ is nilpotent of degree $\leq f(m,n).$
	\begin{proof}
		Consider a commutator \[c=[[a,b_{11},\cdots,b_{1t_1}],[a,b_{21},\cdots,b_{2t_2}],\cdots,[a,b_{q1},\cdots,b_{qt_q}]],\]
		where $b_{ij}\in S,\ t_i\leq m,\ 1\leq i\leq q,\ q=f(m,n).$
		Our aim is to show that $c=0.$
		
		Let an element $b\in S$ occur in the commutator $c$ $|b|$ times.  Clearly $\sum\limits_{b\in S} |b|=t_1+\cdots+t_q.$  Choose $|b|$ new elements $x_{b,1},\cdots,x_{b,|b|}$ in $\til L$ and replace all $|b|$ occurrences of $b$ in $c$ by the symmetrized sum in \linebreak$x_{b,1},\cdots,x_{b,|b|}:$
		\[\cdots\underbrace{b\cdots b}_{|b|}\cdots\rightarrow \sum\limits_{\sigma\in S_{|b|}}\cdots x_{b,\sigma(1)}\cdots x_{b,\sigma(|b|)}\cdots.\]
		We will get a new expression $c'$ in $a,x_{b,j}$'s$,\ b\in S,\ 1\leq j\leq |b|.$
		Denote it $c'=c'(a,x_{b,1},\cdots,x_{b,|b|})$. To show that $c=0$ it is sufficient to show that $c'=0.$ Indeed, let $b=b_1+\cdots+b_k$ be the standard decomposition of $b$. If $k<|b|$, then $c$ is a sum of expressions, each containing one of the elements $b_1,\cdots, b_k$ at least twice. If $k\geq|b|$, then $c=\sum c'(a, b_{i_1},\cdots, b_{i_{|b|}})$, where the sum runs over all $|b|$-element subsets of $\{b_1,\cdots, b_k\}$.
		
		Let $x_{bj}=\sum\limits_\pi x_{bj\pi},\ a=\sum\limits_\pi a_\pi$ be the standard decompositions, \linebreak $x_{bj\pi}=x_{bj\pi}'\otimes e_\pi,\ a_\pi=a_\pi'\otimes e_\pi,\ x_{bj\pi}'\in L, a_\pi'\in I.$
		
		Let $I(b,j)$ be the ideal of the Lie algebra $L$ generated by the subset $\{x_{bj\pi}\}_\pi.$  Suppose at first that for arbitrary $b\in S,\ i\leq j \leq |b|$ we have
		\begin{equation}\label{Star2}
		[I(b,j),I(b,j)]=(0).
		\end{equation}
		Then we can consider the automorphism
		\[A(x_{b,j})=\text{Id}+\sum\limits_{k\geq 1}\adx[k][](x_{b,j}),\ 1\leq j\leq|b|\] and the automorphism $A(b)=A(x_{b,1})\cdots A(x_{b,|b|}).$  By \Cref{Lemma20} for arbitrary elements $a_1,\cdots,a_{|b|}\in\til L$ we have
		\begin{align*}
		&a_1A(b)\otimes\cdots\otimes a_{|b|}A(b)-\sum a_1A(b)\otimes\cdots\otimes a_i\otimes\cdots\otimes a_{|b|}A(b)+\\
		&\sum a_1A(b)\otimes\cdots a_i\otimes\cdots\otimes a_j\otimes\cdots a_{|b|}A(b)-\cdots=\\
		&\sum[a_1,x_{b,\sigma(1)}]\otimes\cdots\otimes[a_{|b|},x_{b,\sigma(|b|)}].
		\end{align*}
		Replacing each symmetric set $x_{b,1},\cdots,x_{b,|b|}$ in the element $c'$ by expressions of the left hand side types we can represent $c'$ as a linear combination of commutators
		\[[a\phi_{11}\cdots\phi_{1m},\cdots,a\phi_{q1}\cdots\phi_{qm}],\]
		where each $\phi_{ij}$ is one of the automorphisms $A(b_1),\cdots,A(b_n),Id.$  There are $\leq (n+1)^m$ elements $a\phi_{i1},\cdots\phi_{im}$ and all of them are sandwiches in the algebra $\til I.$  By the choice of $q=KZ((n+1)^m)$ we conclude that $[a\phi_{11},\cdots,\phi_{1m},\cdots,a\phi_{q1}\cdots\phi_{qm}]=0.$
		
		Now we will drop the assumption that 
		\[[I(b,j),I(b,j)]=(0).\]
		Let $\Lie\langle X\rangle$ be the free Lie algebra on the set of free generators $X=\{ \{x'_{b,j,\pi}\}_\pi,\ \{a_\pi\}_\pi\}.$  Let $I(b,j)$ be the ideal of $\Lie\langle X\rangle$ generated by the set $\{x'_{b,j,\pi}\}_\pi$.  Let $I=\sum\limits_{b,j}[I(b,j),I(b,j)].$  Let $J$ be the ideal of $\Lie\langle X\rangle$ generated by all relations needed to make $a=\sum\limits_\pi a'_\pi e_\pi$ a sandwich in the ideal generated by $a$.  So, $J$ is generated by
		\[[a'_\pi,\rho,a'_\tau]+[a'_\tau,\rho,a'_\pi],[a'_\pi,\rho_1,\rho_2,a'_\tau]+[a'_\tau,\rho_1,\rho_2,a'_\pi],\]
		where $\rho,\rho_1,\rho_2$ are arbitrary commutators in $X$ involving at least one element $a'_\mu.$
		
		Consider the Lie algebra $L=\Lie\langle X\rangle/I+J.$  Since this algebra satisfies condition \ref{Star2} the element $c'$ computed in the algebra $\til L$ lies in $(I+J)\otimes E.$  The ideals $I$ and $J$ are graded with respect to each generator.  The element $c'$ has total degree one with respect to variables $\{x'_{b,j,\pi}\}_\pi$ for each $b,j.$  Since the ideal $I$ does not contain homogeneous elements having degree one with respect to all $\{x'_{b,j,\pi}\}_\pi$ it follows that $c'\in J\otimes E.$  This finishes the proof of the lemma.		
	\end{proof}
\end{lemma}

Now, let $A$ be an associative enveloping algebra of the algebra $L,\ L\subseteq A^{(-)},\ \til L=L\otimes_FE,\ \til A=A\otimes_FE.$

\begin{lemma}
	\label{Lemma22}
	Let $a\in\til L$ be a sandwich.  Then for an arbitrary element $b\in \til L$ we have $[a,b]^p=0.$
	\begin{proof}
		N. Jacobson \cite{J1} noticed that
		 \[\{x_1,\cdots,x_p\}=\sum\limits_{\sigma\in S_p}x_{\sigma(1)}\cdots x_{\sigma(p)}=\sum\limits_{\sigma\in S_{p-1}}[x_p,x_{\sigma(1)},\cdots,x_{\sigma(p-1)}].\]
		 Let $b=\sum\limits_\pi b_\pi$ be the standard decomposition of the element $b.$  Then $[a,b]=\sum\limits_\pi [a,b_\pi]$ and $[a,b]^p=\sum\{[a,b_{\pi_1}],\cdots,[a,b_{\pi_p}]\}.$  Each summand on the right hand side is equal to zero since $[[a,b_{\pi_i}],[a,b_{\pi_j}]]=0,$ which proves the lemma.
	\end{proof}
\end{lemma}

\begin{lemma}
	\label{Lemma23}
	Let $I$ be an ideal of the Lie algebra $L.$  Let $a\in \til I$ be a sandwich in $\til I$ such that $a^p=0$ in the algebra $\til A$.  Let $S\subset \til L$ be a finite set of $\leq n$ elements.  Then the associative subalgebra of $\til A$ generated by commutators $[a,b_1,\cdots,b_t]$, where $b_i\in S,\ t\leq m,$ is nilpotent of degree $\leq p^{(n+1)^{mf(m,n)}}.$
	\begin{proof}
		Denote $q=p^{(n+1)^{mf(m,n)}}.$  We need to show that an arbitrary product
		\[[a,b_{11},\cdots,b_{1t_1}]\cdots[a,b_{q1},\cdots,b_{qt_q}],\]
		where $b_{ij}\in B,\ t_i\leq m,\ 1\leq i\leq q,$ is equal to 0.
		
		Arguing as in the proof of \Cref{Lemma21} we can reduce the problem to showing that an arbitrary product $(a\phi_{11}\cdots\phi_{1m})\cdots(a\phi_{q1}\cdots\phi_{qm})=0,$ where $\phi_{ij}\in \Aut\til A,\ \phi_{ij}(\til L)=\til L,\ \phi_{ij}(\til I)=\til I,\ \#\{\phi_{ij}|1\leq i\leq q, 1\leq j\leq m\}\leq n+1,$.
		
		Let $Y=\{a\phi_{i1}\cdots\phi_{im},1\leq i\leq q\},\ |Y|\leq (n+1)^m.$  An arbitrary element $y$ from $Y$ is a sandwich in $\til I$ and $y^p=0$ in $\til A.$
		
		Let $L_1$ be the Lie algebra generated by $Y.$  By \Cref{Lemma21} $L_1^{f(m,n)}=(0).$  Let $\rho_1,\cdots,\rho_r$ be left normed commutators in $Y$ that form a basis of $L_1,\ r\leq |Y|^{f(m,n)}\leq(n+1)^{mf(m,n)}.$
		
		By \Cref{Lemma22} for each commutator $\rho_i$ we have $\rho_i^p=0.$  Now the Poincare-Birkhoff-Witt theorem implies the assertion of the lemma.
	\end{proof}
\end{lemma}

\section{Proof of \Cref{Theorem1}}
\label{Section6}

Let a Lie algebra $L$ over a field $F$ of characteristic $p>0$ and its associative enveloping algebra $A$, $L\subseteq A^{(-)},\ A=\langle L\rangle,$ satisfy the conditions outlined at the beginning of $\mathsection$3:
\begin{enumerate}[1)]
	\item $L$ is a graded Lie algebra generated by elements $x_1,\cdots,x_m$ of degree 1; every element from the Lie set $S=S\langle x_1,\cdots,x_m\rangle$ is ad-nilpotent;
	\item $L$ satisfies a polynomial identity;
	\item the grading of $L$ extends to $A$; both algebras $L$ and $A$ are graded just infinite.
\end{enumerate}

Recall that an element $g(x_1,\cdots,x_r)$ of the free associative algebra is called a weak identity of the pair $(L,A)$ if $g(a_1,\cdots,a_r)=0$ for arbitrary elements $a_1,\cdots,a_r\in L.$  In particular, every Lie identity of the algebra $L$ can be viewed as a weak identity of the pair $(L,A).$

Let $k$ be a minimal degree of a nonzero weak identity satisfied by $(L,A).$  Without loss of generality we can assume that $(L,A)$ satisfies a weak identity
\[h(x_1,\cdots,x_k)=x_1\cdots x_k+\sum\limits_{1\neq\sigma\in S_k}\alpha_\sigma x_{\sigma(1)}\cdots x_{\sigma(k)}.\]
We remark that the pair $(\til{L}, \til{A})$ satisfies this weak identity as well.

Let $h(x_1,\cdots,x_k)=\sum\limits_{i=1}^k h_i(x_1,\cdots,\hat{x_i},\cdots,x_k)x_i.$  The ideal $M$ of the algebra $A$ generated by all values of $h_k(a_1,\cdots,a_{k-1}),\ a_i\in L$ is graded nonzero and therefore has finite codimension in $A$. Hence there exists $d\geq1$ such that $A^d\subseteq M.$  Denote $\widehat{A}=A+F 1,$

\begin{lemma}
	\label{Lemma24}
	For an arbitrary element $a\in \til{L}$ we have 
	\[A^da\subseteq\sum x_{i_1}\cdots x_{i_t}a\widehat{A},\ t\leq d-1.\]
	\begin{proof}
		For an arbitrary product of length $d$ we have 
		\[x_{i_1}\cdots x_{i_d}=\sum\limits_j \alpha_j v'_j h_{k}(\rho_{j1},\cdots,\rho_{j,k-1})v''_j,\]
	 where $\alpha_j\in F;\ v'_j,v''_j,$ and $\rho_{j1},\cdots,\rho_{j,k-1}$ are monomials and commutators in generators $x_1,\cdots,x_m$ of total length $d$.
		
		Let $v''_j=x_{\mu_1}\cdots x_{\mu_r}.$  Then
		\[v''_ja=a'+\sum\limits_t \pm w'_{jt}aw''_{jt},\]
		where $a'=[x_{\mu_1},[x_{\mu_2},[\cdots,[x_{\mu_r},a]\cdots];\ w'_{jt},w''_{jt}$ are products in generators of total degree equal to the degree of $v''_j$ and the products $w''_{jt}$ are not empty.  Hence,
		\[\begin{multlined}
		x_{i_1}\cdots x_{i_d}a=\sum\limits_j\alpha_jv'_jh_{k}(\rho_{j,1},\cdots,\rho_{j,k-1})a'+\\
		\sum\limits_{j,t}\alpha_jv_j'h_{k}(\rho_{j,1},\cdots,\rho_{j,k-1})w_{jt}'aw_{jt}''.
		\end{multlined}\]
		Furthermore,
		\[\begin{multlined}h_{k}(\rho_{j1},\cdots,\rho_{j,k-1})a'=h(\rho_{j1},\cdots,\rho_{j,k-1},a')-\\
		\sum\limits_{i=1}^{k-1}h_i(\rho_{j1},\cdots,a',\cdots,\rho_{j,k-1})\rho_{ji}=-\sum\limits_{i=1}^{k-1}h_{i}(\rho_{j1},\cdots,a',\cdots\rho_{j,k-1})\rho_{ji}.
		\end{multlined}\]
		We proved that $x_{i_1}\cdots x_{i_d}a$ is a linear combination of elements $w'aw'',$ where $w',w''$ are products in $x_1,\cdots,x_m,$ with the length of $w'$ less than $d.$
	\end{proof}
\end{lemma}

Consider the function $g(m,n)=p^{(n+1)^{mf(m,n)}}$.

\begin{lemma}
	\label{Lemma25}
	Let $I$ be an ideal of the algebra $L$.  Let $a$ be a sandwich in $\til I$ and $a^p=0$ in $\til A.$  Let $a'=[a_1,\cdots,a_d],$ where $a_i=[a,u_{i1},\cdots,u_{it_i}],\ u_{ij}\in\til L,\ t_i\geq0,\ i=1,\cdots,d.$  Let $G=g(\max\limits_i\{ t_i+d-1\},t_1+\cdots+t_d+d^{d+1}\}).$  Then $(Aa')^G=(0).$ In particular, an arbitary element from $\text{Id}_{\til{L}}(a')^d$ generates a nilpotent ideal in $\til{A}.$
	\begin{proof}
		Suppose that $a'v_1a'\cdots v_{G-1}a'\neq0,$ where $v_1,\cdots,v_{G-1}$ are products in generators $x_1,\cdots,x_m$ of lengths $l(v_1),\cdots,l(v_{G-1})$ respectively.  Without loss of generality we will assume that the vector of lengths $(l(v_1),\cdots,l(v_{G-1}))$ is lexicographically minimal.  By \Cref{Lemma24} this implies that $l(v_i)\leq d-1,\ 1\leq i\leq G-1.$
		
		If $v=x_{j_1}\cdots x_{j_r}$ then we denote 
\[[va']=[x_{j_1},[x_{j_2},[\cdots[x_{j_r},a']\cdots]=(-1)^r[a',x_{j_r},x_{j_{r-1}},\cdots,x_{j_1}].\]
		Again by lexicographical minimality we have 
		\[a'v_1a'\cdots v_{G-1}a'=a'[v_1a']\cdots[v_{G-1}a'].\]

Lemma \ref{Lemma23} is not applicable to the sandwich $a'$ and elements $x_1,\cdots, x_m$ because $x_1,\cdots, x_m$ do not lie in $\til{L}$. However, for an arbitrary word $v=x_{j_1}\cdots x_{j_r}$, $r<d$, we have
\[[va']=\sum\limits_i\left[[v_{i1}a_1],\cdots,[v_{id}a_d]\right],\]
where $v_{i1},\cdots, v_{id}$ are words in $x_1,\cdots, x_m$ of total length $r<d$. Hence, at least one of these words is empty. Now we can apply Lemma \ref{Lemma23} to the sandwich $a$ and the set $S=\{[a,u_{i1},\cdots, u_{it_i},x_{j1},\cdots, x_{jr}]|1\leq i\leq d, 1\leq j_1,\cdots, j_r\leq m, 0\leq r<d\}\subset\til{L}$. By Lemma \ref{Lemma23} and the choice of the number $G$, we have $a'[v_1a']\cdots[v_{G-1}a']=0$, which proves the lemma.
	\end{proof}
\end{lemma}

Now our aim is the following:
\begin{proposition}
	\label{Proposition4}
	There exist integers $N\geq1,\ s\geq 1$ and a regular divided polynomial $v$ defined on $L^s$ such that every value of $v$ on $\til L^s$ generates a nilpotent ideal in $\til A$ of degree $\leq N.$
\end{proposition}

Suppose that the algebra $L$ satisfies an identity $f(x_0,x_1,\cdots,x_{n-1})=[x_0,x_1,\cdots,x_{n-1}]+\sum\limits_{1\neq\sigma\in S_{n-1}}\alpha_\sigma[x_0,x_{\sigma(1)},\cdots,x_{\sigma(n-1)}],$ where $\alpha_\sigma\in F$ and $n$ is the minimal degree of an identity satisfied by $L.$

Consider the following element of degree $n-1$: 
\begin{align*}
	&f_1(x_0,x_2,\cdots,x_{n-1})=[x_0,x_2,\cdots,x_{n-1}]+\\ 
	&\sum\limits_{\substack{1\neq\sigma\in S_{n-1}\\ \sigma(1)=1}} \alpha_\sigma[x_0,x_{\sigma(2)},\cdots,x_{\sigma(n-1)}]=x_0H(\ad(x_2),\cdots,\ad(x_{n-1})),
\end{align*}
where $H(y_2,\cdots,y_{n-1})=y_2\cdots y_{n-1}+\sum\limits_{\substack{1\neq\sigma\in S_{n-1}\\ \sigma(1)=1}}\alpha_\sigma y_{\sigma(2)}\cdots y_{\sigma(n-1)}.$

\begin{lemma}
	\label{Lemma26}
	Let $w(x_1,\cdots,x_r)$ be a regular divided polynomial defined on $L^s.$  Then the divided polynomial 
	\[w'(x_1,\cdots,x_r,y_2,\cdots,y_{n-1})=f_1(w(x_1,\cdots,x_r),y_2,\cdots,y_{n-1})\]
	 is also regular.
	 \begin{proof}
	 	Choose $t\geq s$ and elements $a_1,\cdots,a_r\in\til L^t$ such that $a'=w(a_1,\cdots,a_r)\neq 0.$  Suppose that for arbitrary elements $b_2,\cdots,b_{n-1}\in\til L^t$ we have $f_1(a',b_2,\cdots,b_{n-1})=0.$  Let $a'=\sum\limits_\pi a'_\pi$ be the standard decomposition, $a'_\pi=\bar{a'_\pi}\otimes e_\pi,\ \bar{a'_\pi}\in L^t.$  Then the assumption above means that for an arbitrary $\pi$ we have $f_1(\bar{a'_\pi},L^t,\cdots,L^t)=(0).$  Choose $\pi$ such that $\bar{a'_\pi}\neq 0.$  Let $R=\langle \text{Id},\ad(x),x\in L\rangle\subseteq\End_F(L)$ be the multiplication algebra of the algebra $L.$  Consider the ideal $\text{id}_L(\bar{a'_\pi})=\bar{a'_\pi}R$ generated by the element $\bar{a'_\pi}$ in $L.$
	 	
	 	For an arbitrary element $x\in L$ we have 
	 	\[[\ad(x),H(ad(L^t),\cdots,\ad(L^t))]\subseteq H(\ad(L^t),\cdots,ad(L^t)).\]	 
	 	Hence,
	 	\[RH(ad(L^t),\cdots,\ad(L^t))\subseteq H(\ad(L^t),\cdots,ad(L^t))R.\]
	 	Now for an arbitrary operator $P\in R$ we have
	 	\[\begin{multlined}f_1(\bar{a'_\pi}P,L^t,\cdots,L^t)=\bar{a'_\pi}PH(L^t,\cdots,L^t)\\
	 	\subseteq\bar{a'_\pi}H(L^t,\cdots,L^t)R=f_1(\bar{a'_\pi},L^t,\cdots,L^t)R=(0).\end{multlined}\]
	 	Since the algebra $L$ is graded just infinite it follows that $\text{id}_L(\bar{a'_\pi})\supseteq L^k$ for sufficiently large $k\geq t.$  We proved that $f_1(L^k,\cdots,L^k)=(0).$
	 	
	 	The algebra $L^k$ is finitely generated by \Cref{Lemma3}.  By the induction assumption on the degree of the identity $f$ we conclude that the Lie algebra $L^k$ is finite dimensional.  Therefore the algebra $L$ is finite dimensional as well which contradicts our assumption that the algebra $L$ is graded just infinite, proving the lemma.
	 \end{proof}
\end{lemma}

\begin{corollary}\label{Corollary1}
	Let $q\geq 1.$  Choose $r+q(n-2)$ distinct variables $x_1,\cdots,x_r,\ y_{ij},\ 1\leq i\leq q,\ 2\leq j\leq n-1.$  Then the divided polynomial
	\[\begin{multlined}w_q=w(x_1,\cdots,x_r)H(\ad(y_{11}),\cdots,\ad(y_{1,n-1}))\cdots\\ H(\ad(y_{q1},\cdots,\ad(y_{q,n-1}))\end{multlined}\]
is regular.
\end{corollary}

\begin{lemma}
	\label{Lemma27}
	For arbitrary elements $a,b_2,\cdots,b_{n-1},c\in\til L$ we have
	\[[f_1(a,b_2,\cdots,b_{n-1}),c]\in\sum F[a,b_{i_2},\cdots,[b_{i_k},c],\cdots,b_{i_{n-1}}],\]
where $i_2,i_3,\cdots,i_{n-1}$ is a permutation of $2,\cdots,n-1;\ 2\leq k\leq n-1.$

\begin{proof}
	We have 
	\[\begin{multlined}aH(\ad(b_2),\cdots,\ad(b_{n-1})\ad(c)=a\ad(c)H(\ad(b_2),\cdots,\ad(b_{n-1}))+\\
	\sum aH(\ad(b_2),\cdots,\ad([b_k,c]),\cdots,\ad(b_{n-1})).\end{multlined}\]
Let us represent the polynomial $f$ of minimal degree as
\[f(x_0, x_1, \cdots, x_{n-1})=\sum\limits_{i=1}^{n-1}x_0\ad(x_i)H_i(\ad(x_1), \cdots, \widehat{\ad(x_i)}, \cdots, \ad(x_{n-1})),\]
where $H_1=H$. Then
\begin{align*}
&a\ad(c)H(\ad(b_2), \cdots, \ad(b_{n-1}))=-c\ad(a)H(\ad(b_2), \cdots, \ad(b_{n-1}))\\
&=\sum\limits_{i=2}^{n-1}[c,b_i]H_i(\ad(a), \ad(b_2), \cdots, \widehat{\ad(b_i)}, \cdots, \ad(b_{n-1}))\\
& \in \sum F[a,b_{i_2}, \cdots, [b_{i_k},c], \cdots, b_{i_{n-1}}].
\end{align*}
\end{proof}
\end{lemma}

\begin{lemma}
	\label{Lemma28}
	Let $\nu\geq 1.$  Suppose that a divided polynomial $w(x_1,\cdots,x_r)$ is defined on $L^s$ and for all $t\leq q(n-2)+\nu$ the divided polynomial $[w(x_1,\cdots,x_r),y_1,\cdots,y_t,w(x_1,\cdots,x_r)]$ is identically zero on $\til L^s.$  Then for arbitrary elements $a_k,b_{ij}\in\til L^s$ we have
	\[\begin{multlined}
	[ w _q(a_k,b_{ij},1\leq k\leq r,\ 1\leq i\leq q,\ 2\leq j\leq n-1),\\
	\underbrace{L,L,\cdots,L}_{\mu},\underbrace{L^s,\cdots,L^s}_{\nu},w(a_1,\cdots,a_r)]=(0)\end{multlined}\]
	for $\mu\leq q.$
	\begin{proof}
		Applying \Cref{Lemma27} $q$ times we get
		\[[w_q(a_k,b_{ij}),\underbrace{L,\cdots,L}_\mu]\subseteq[w(a_1,\cdots,a_r),\underbrace{\til L^s,\cdots,\til L^s}_{q(n-2)}],\]
		which implies the assertion of the lemma.
	\end{proof}
\end{lemma}

\begin{lemma}
	\label{Lemma29}
	Let $a\in \til L$ and $[a,\underbrace{\til L,\cdots,\til L}_\mu,a]=(0)$ for $\mu\leq2d.$  Then for an arbitrary element $b\in \til L$ the commutator $[a,b]$ generates a nilpotent ideal in $\til A$ of degree $\leq m^d(p-1)+1.$
	\begin{proof}
		Recall that the algebra $L$ is generated by $m$ elements $x_1,\cdots,x_m.$  Suppose that $[a,b]v_1[a,b]\cdots v_{N-1}[a,b]\neq0,$ where $v_i$ are products of the generators and the vector of lengths $(l(v_1),\cdots,l(v_{N-1})$ is lexicographically minimal among all vectors with this property.  Then by \Cref{Lemma24} $l(v_i)<d$ for $i=1,\cdots,N-1.$
		
		As above, for a product $v=x_{i_1},\cdots,x_{i_k}$ we denote
		\[[v[a,b]]=[x_{i_1}[x_{i_2},[\cdots[x_{i_k},[a,b]]\cdots].\] 
		 We have
		 \[[a,b]v_1\cdots v_{N-1}[a,b]=[a,b][v_1,[a,b]]\cdots[v_{N-1}[a,b]].\]
		 By our assumption the commutators $[a,b],\ [v_i[a,b]],\ 1\leq i\leq N-1,$ commute.  There are $<m^d$ such commutators.  Hence at least one commutator $[v_i[a,b]]$ occurs $\geq p$ times.  If $b=\sum\limits_\pi b_\pi$ is the standard decomposition then $[v_i [a,b]]^p$ is a sum of expressions 
		 \begin{align*}
		 &\{[v_i[a,b_{\pi_0}]],\cdots,[v_i[a,b_{\pi_{p-1}}]]\}=\\
		 &\sum\limits_{\sigma\in S_{p-1}}[[v_i,[a,b_{\pi_0}]],[v_i[a,b_{\pi_{\sigma(1)}}]],\cdots,[v_i[a,b_{\pi_{\sigma(p-1)}}]]]=0.
		\end{align*}
		Hence $[v_i[a,b]]^p=0,$ which finishes the proof of the lemma.
	\end{proof}
\end{lemma}

\begin{lemma}
	\label{Lemma30}
	For an arbitrary sandwich of the algebra $\til L^s$ we have $a^p=0$ in $\til A.$
	\begin{proof}
		Recall that the algebra $A$ is a homomorphic image of the subalgebra $\langle \ad(x),x\in L\rangle\subseteq End_F(L).$ 
		
		Hence it is sufficient to show that $\ad(a)^p=0$ in $\til L.$  If $p\geq 3$ then $\til L\ad(a)^p\subseteq[\til L^s,a,a]=(0).$  Let $p=2.$  We have $[\til L^s,a,a]=(0).$  Since the mapping $\til L\rightarrow\til L,\ x\rightarrow [x,a,a]$ is a derivation of $\til L$ it follows that $[\til L,a,a]$ lies in the centralizer of $\til L^s.$  In a graded just infinite algebra $L$ the centralizer of $L^s$ is zero.  Hence $[\til L,a,a]=(0)$ and again $\ad(a)^2=0$, which finishes the proof of the lemma.
			\end{proof}
		\end{lemma}
		
\begin{proof}[Proof of \Cref{Proposition4}]
	We will start with the regular divided polynomial $v(x_1,\cdots,x_r)$ of \Cref{Proposition3} defined on $L^s.$  Every value of $v$ on $\til L^s$ is a sum of $t$ sandwiches of the algebra $\til{L}^s.$
	
	We will construct a sequence of finite sets $M_i$ of divided polynomials defined on $L^s,\ i\geq0.$
	
	Let $M_0=\{v\},\ M_{i+1}=\{[w,y_1,\cdots,y_\mu,w]|w\in M_i,\ \mu\leq4d(n-2);\ y_1,\cdots,y_\mu \text{ are variables not involved in } w\}.$
	
	Let $2^i\geq (d-1)t+1,$ $w \in M_i.$ Let $b$ be a value of the divided polynomial $v$ on $\til{L}^s.$ By Proposition \ref{Proposition3}, $b=b_1 + \cdots + b_t,$ where for $1 \leq j \leq t,$ each element $b_j$ is a sandwich of the algebra $\til{L}^s$. Since $2^i \geq (d-1)t +1,$ it follows that the value $c$ of the divided polynomial $w$ that is obtained by iterating the value $b$ of $v$ lies in $\sum\limits_{j=1}^{t}\text{id}_{\til{L}}(b_j)^d.$ By Lemma \ref{Lemma25} there exists an integer $G \geq 1$ such that every value of the divided polynomial $w$ on $\til{L}^s$ generates a nilpotent ideal in $\til{A}$ of degree $\leq G.$

	
	If at least one divided polynomial in $M_i$ is regular, then we are done.
	
	If none of the divided polynomials in $M_i$ is regular and $i$ is the minimal integer with this property then there exists a regular divided 
	polynomial $w(x_1,\cdots,x_r)$ defined on $\til L ^s$ and an integer $t\geq s$ such that all the divided polynomials
	\[[w(x_1,\cdots,x_r),y_1,\cdots,y_\mu,w(x_1,\cdots,x_r)],\ 1\leq\mu\leq 4d(n-2),\]
	are identically zero on $\til L^t.$
	
	Consider the regular divided polynomial $w_{2d}(x_1,\cdots,x_r,y_{ij},1\leq i\leq2d,\ 2\leq j\leq n-1).$  By \Cref{Lemma28} we have
	\[[w_{2d}(a_1,\cdots,a_r,b_{ij}),\underbrace{L,L,\cdots,L}_\mu,\underbrace{L^t,\cdots,L^t}_\nu,w(a_1,\cdots,a_r)]=(0),\]
	$ \mu\leq2d,\ \nu\leq 2d(n-2)$ for arbitrary elements $a_1,\cdots,a_r,b_{ij}\in\til L^t$. We have $w_{2d}(a_1, \cdots, a_r, b_{ij})\in \sum\limits_{\nu=1}^{2d(n-2)}[w(a_1, \cdots, a_r), \underbrace{L^t, \cdots, L^t}_\nu].$
Therefore
	\[[w_{2d}(a_1,\cdots,a_r,b_{ij}),\underbrace{L,L,\cdots,L}_\mu,w_{2d}(a_1,\cdots,a_r,b_{ij})]=(0)\]
	for $\mu\leq 2d$.
	
	The divided polynomial \[w'_{2d}(x_0,x_1,\cdots,x_r,y_{ij})=[w_{2d}(x_1,\cdots,x_r,y_{ij}),x_0]\] 
	 is regular.  Indeed, regularity of $w_{2d}$ has been established in Corollary \ref{Corollary1}. If the divided polynomial $w_{2d}'$ vanishes on some power of $L$, then some power of $L$ has a nonzero centralizer. Since the algebra $L$ is graded just infinite, it follows that this centralizer is of finite codimension. Hence the algebra $L$ is solvable and therefore finite dimensional, a contradiction. By \Cref{Lemma29} an arbitrary value of $w'_{2d}$ on $\til L^t$ generates a nilpotent ideal in $\til A$ of degree $\leq m^d(p-1)+1.$  This finishes the proof of \Cref{Proposition4}.
\end{proof}

By \Cref{Proposition4} there exist integers $s\geq 1,\ N_1\geq1$ and a regular divided polynomial $w$ defined on $L^s$ such that every value of $w$ on $\til L^s$ generates a nilpotent ideal in $\til A$ of degree $\leq N_1.$  Let $f(x_1,x_2,\cdots,x_r)\in\til L\ast\Lie\langle X\rangle$ be the linearization of the divided polynomial $w$.  If $l$ is the degree of $w$ then every value of $f$ in $\til L^s$ is a linear combination of $2^l$ values of $w$.  Hence for arbitrary $a_1,a_2,\cdots a_r \in\til L^s$ the element $f(a_1,\cdots,a_r)$ generates a nilpotent ideal in $\til A$ of degree $\leq N_2=2^l(N_1-1)+1.$

Choose variables $x_{ij} \in X$, $1\leq i\leq r$, $1 \leq j \leq N_2$. The pair $(L,A)$ satisfies the system of weak identities
\begin{align*}
&F_{N_2}=\{F_{N_2}(x_{ij},y_k)=\sum\limits_{\sigma_1,\cdots, \sigma_r\in S_{N_2}}f(x_{1\sigma_1(1)}, x_{2\sigma_2(1)},\cdots, x_{r\sigma_r(1)})y_1\\
&f(x_{1\sigma_1(2)}, x_{2\sigma_2(2)},\cdots, x_{r\sigma_r(2)})y_2\cdots y_{N_2-1}f(x_{1\sigma_1(N_2)}, x_{2\sigma_2(N_2)},\cdots, x_{r\sigma_r(N_2)})\\
&=0\},
\end{align*}
where $x_{ij}$ take values in $L^s$ and $y_k$'s are arbitrary products of independent variables taking values in $L$. Thus values of $y_k$'s lie in $A$.

Indeed, these weak identities are satisfied by the pair $(\til{L}, \til{A})$. It remains to notice that the pair $(L,A)$ and $(\til{L},\til{A})$ satisfy the same multilinear weak identities. 

Let $N$ be the minimal integer such that $(L,A)$ satisfies the weak identities $F_N=0$.

Let $a=F_{N-1}(a_{ij},b_k)\neq 0$, $a_{ij}\in L^s$, $1\leq i \leq r$, $1\leq j\leq N-1$; $b_k\in A$. Choose $t > \max\limits_{i,j}\deg({a_{ij}})$. By regularity of the polynomial $f$, there exist homogeneous elements $a_1, \cdots, a_r \in L^t$ such that $f(a_1, \cdots, a_r) \neq 0$. The identities $F_N=0$ immediately imply the following lemma.

%
%
%

\begin{lemma}
	\label{Lemma31}
	Let $b_\mu,c_\mu\in A+F1.$  Then $(\sum\limits_\mu b_\mu f(a_1,\cdots,a_r)c_\mu)a$ is a linear combination of elements $(\sum\limits_\mu b_\mu f(a_1',\cdots,a_r')c_\mu)a',$ where $a'\in A,\ a_1'\in\{a_1,a_{1k},\ 1\leq k\leq N-1\},\ a_2'\in \{a_2,a_{2k},\ 1\leq k\leq N-1\},\cdots,$ and at least one $a_i'$ lies in $\{a_{ik},\ 1\leq k\leq N-1\}.$
\end{lemma}

The ideal $I$ generated by $f(a_1,\cdots,a_r)$ in $A$ has finite codimension.  Let $A^l\subseteq I.$ 

\begin{corollary}\label{CorollaryL31}
	Let $u$ be a homogeneous element of degree $\geq l.$  Then for arbitrary $b_\mu,c_\mu\in A+F1$ the element $(\sum\limits_\mu b_\mu u c_\mu)a$ is a linear combination of elements $(\sum\limits_\mu b_\mu u'c_\mu)a',$ where $u'\in A$ are homogeneous elements, $\deg u'<\deg u.$
\end{corollary}

Indeed, from $u\in A^\ell \subseteq I$ it follows that $u=\sum\limits_\nu b_\nu' f(a_1, \cdots, a_r)c_\nu'$, where $b_\nu', c_\nu'$ are homogeneous elements, $\deg b_\nu'+\deg c_\nu'+\sum\limits_{i=1}^r\deg a_i=\deg u$. Lemma \ref{Lemma31} implies that $(\sum\limits_{\mu,\nu}b_\mu b_\nu' f(a_1, \cdots, a_r)c_\nu' c_\mu)a$ is a linear combination of elements $(\sum\limits_{\mu,\nu} b_\mu b_\nu' f(a_1', \cdots, a_r')c_\nu' c_\mu )a',$ where\linebreak $\deg a_i' \leq \deg a_i$ for all $i$ and at least for one $i$, we have $\deg a_i' < \deg a_i$. Hence for $u'=\sum\limits_\nu b_\nu' f(a_1', \cdots, a_r')c_\nu '$ we have $\deg u' < \deg u.$

\begin{lemma}
	\label{Lemma32}
	Let $h(x_1,\cdots,x_q)$ be a multilinear element of the free associative algebra such that for arbitrary homogeneous elements \linebreak$u_1,\cdots,u_q\in A$ of degrees $\deg u_1<l,\cdots,\deg u_q<l$ we have \linebreak$h(u_1,\cdots,u_q)=0.$  Then $h=0$ holds identically on $A.$
	\begin{proof}
		Let $v_1,\cdots,v_q\in A$ be homogeneous elements of $A$ such that $h(v_1,\cdots,v_q)\neq0.$  Let the total degree $\sum\limits_{i=1}^q\deg(v_i)$ be minimal among all $q$-tuples with this property.  At least one element $v_i$ has degree $\geq l.$

	Let us show that the graded just infinite algebra $A$ is graded prime.  Indeed, if $I_1,I_2$ are nonzero graded ideals of $A$ then $A^{t_1}\subseteq I_1,A^{t_2}\subseteq I_2$ for some integers $t_1,t_2\geq 1.$  If $I_1I_2=(0)$ then $A^{t_1+t_2}=(0),$ a contradiction.
		
		Hence there exists an element $b\in A$ such that $h(v_1,\cdots,v_q)ba\neq0.$
		
		By \Cref{CorollaryL31} the element $h(v_1,\cdots,v_q)ba$ is a linear combination of elements $h(v_1	,\cdots,v_{i-1},v',v_{i+1},\cdots,v_q)ba',$ where $\deg v'<\deg v_i.$  This contradicts the minimality of $\sum\limits_{i=1}^q\deg(v_i)$ and finishes the proof of the lemma.
	\end{proof}
\end{lemma}

\begin{remark*}
A nonzero element $h(x_1,\cdots, x_q)$ satisfying the hypothesis of Lemma \ref{Lemma32} exists for an arbitrary finitely generated algebra. Moreover, for an arbitrary associative algebra $A$ and a finite dimensional subspace $V\subset A$, there exists a multilinear element $h(x_1,\cdots, x_q)$ of the free associative algebra such that $h(u_1,\cdots, u_q)=0$ for all elements $u_1,\cdots, u_q\in V$. Indeed, it is sufficient to choose an element that is skew-symmetric in $x_1,\cdots, x_q$, where $q=\dim_FV+1$ , for example the element
\[h(x_1,\cdots, x_q)=\sum\limits_{\sigma\in S_q}(-1)^{|\sigma|}x_{\sigma(1)}\cdots x_{\sigma(q)}.\]
\end{remark*}

\begin{proof}[Proof of \Cref{Theorem1}] 
	
	It is known (see \cite{NO}) that a graded prime algebra is prime.  By \Cref{Lemma32}, $A$ is a PI-algebra.  The prime PI-algebra $A$ has a nonzero center $Z$ and the ring of fractions $(Z\setminus\{0\})^{-1}A$ is a finite dimensional associative algebra over the field $(Z\setminus\{0\})^{-1}Z$ (see \cite{R}).  Now the Engel-Jacobson theorem \cite{J1} implies that the algebra $A$ is nilpotent, a contradiction.  Thus \Cref{Theorem1} is proved.
\end{proof}

\begin{proof}[Proof of \Cref{Theorem2}]
	Let $G$ be a residually-$p$ finitely generated torsion group.  Let $G=G_1\geq G_2\geq\cdots$ be the Zassenhaus filtration.  Consider the Lie algebra $L_p(G)=\bigoplus\limits_{i\geq 1}G_i/G_{i+1}.$  Because of the torsion property of elements of $G$, an arbitrary homogeneous element $a\in G_i/G_{i+1}$ of the Lie algebra $L_p(G)$ is ad-nilpotent (see \cite{K3,VL}).
	
	Consider the subalgebra $L$ of $L_p(G)$ generated by $G_1/G_2.$  If the Lie algebra $L_p(G)$ satisfies a polynomial identity then by \Cref{Theorem1} the finitely generated Lie algebra $L$ is nilpotent.  This implies that the pro-$p$ completion $G_{\hat{p}}$ of the group $G$ is $p$-adic analytic and therefore linear (see \cite{DSMS}).  Now finiteness of $G$ follows from theorems of Burnside and Schur \cite{J3}.
\end{proof}

\begin{proof}[Proof of \Cref{Theorem3}]
	Let $Fr$ be the free pro-$p$ group.  Let $Fr=Fr{_1}>Fr_{2}>\cdots$ be the Zassenhaus filtration of $Fr$.  Suppose that the pro-$p$ completion $G_{\hat{p}}$ satisfies the pro-$p$ identity $w=1,\ w\in Fr{_n}\setminus Fr_{n+1},$ hence $w=\rho_1^{p^{s_1}}\cdots\rho_r^{p^{s_r}}w',$ where each $\rho_i$ is a left normed group commutator of length $l_i,\ p^{s_i}\cdot l_i=n,\ w'\in Fr_{n+1}.$
	
	Considering, if necessary, $[w,x_0]$ instead of $w$, we can assume that $n$ is not a multiple of $p$, and $w=\rho\cdots\rho_rw',$ where all commutators $\rho_1,\cdots,\rho_r$ are of length $n.$
	
	Let $\bar{\rho_i}$ be the commutator from the free Lie algebra that mimics the group commutator $\rho_i$.  Then the Lie algebra $L_p(G)$ satisfies the polynomial identity $\sum\limits_i\bar{\rho_i}=0.$  By \Cref{Theorem2} we conclude that $|G|<\infty.$
\end{proof}

\bibliographystyle{alphanum}
\bibliography{Lie Algebras and Torsion Groups with Identity}

\end{document}